\documentclass[12pt]{article}
\usepackage[margin=1in]{geometry}

\usepackage{amsmath,amssymb}
\usepackage{mathtools}
\usepackage{mathrsfs} 
\usepackage{enumitem}

\usepackage{graphicx} 

\usepackage{tikz-cd}

\usepackage{microtype}
\usepackage{todonotes}

\usepackage{caption, subcaption}
\usepackage{svg} 

\usepackage{amsthm}

\usepackage{hyperref}
\usepackage{cleveref}

\newtheorem{mainthm}{Theorem}

\swapnumbers
\newtheorem{theorem}{Theorem}[section]
\newtheorem{corollary}[theorem]{Corollary}
\newtheorem{lemma}[theorem]{Lemma}
\newtheorem{prop}[theorem]{Proposition}
\newtheorem*{prop*}{Proposition}

\theoremstyle{definition}
\newtheorem{define}[theorem]{Definition}
\newtheorem{example}[theorem]{Example}
\newtheorem{remark}[theorem]{Remark}
\newtheorem{setup}[theorem]{}

\newcommand{\Rd}{\mathbb R^d}

\newcommand{\GM}{\mathcal G_M}
\newcommand{\gM}{\mathfrak g_M}

\newcommand{\TGM}{\mathbb T M} 

\newcommand{\Diff}{\mathrm{Diff}}
\newcommand{\DiffcM}{\mathrm{Diff}_c(M)}
\newcommand{\X}{\mathscr X}
\newcommand{\XcM}{\X_c(M)}
\newcommand{\CcMR}{C_c^\infty(M, \mathbb R)}
\newcommand{\R}{\mathbb R}
\newcommand{\C}{\mathbb C}
\newcommand{\N}{\mathbb N}
\newcommand{\T}{\mathbb T}
\newcommand{\Z}{\mathbb Z}
\newcommand{\cO}{\mu}
\newcommand{\PW}{\mathrm{PW}_\mathrm{fib}}
\newcommand{\Ad}{\mathrm{Ad}}
\newcommand{\ad}{\mathrm{ad}}
\newcommand{\id}{\mathrm{id}}
\newcommand{\ev}{\mathrm{ev}}
\newcommand{\evol}{\mathrm{evol}}
\newcommand{\Fl}{\mathrm{Fl}}
\newcommand{\Lf}{\mathbf{L}}
\newcommand{\ddt}{\frac{d}{dt}}
\newcommand{\fiberF}{\mathcal F_\mathrm{fib}} 
\newcommand{\blangle}{\bigl\langle}
\newcommand{\brangle}{\bigr\rangle}
\newcommand{\arrows}{\rightrightarrows}

\newcommand{\divv}{\mathrm{div}}
\newcommand{\cptops}{\mathcal B_0}
\newcommand{\Qh}{\mathcal Q_h}
\newcommand{\supp}{\mathrm{supp}}
\newcommand{\blank}{{-}}
\newcommand{\interior}{\iota} 
\newcommand{\boundary}{\beta} 
\newcommand{\vol}{V_g}
\newcommand{\tr}{\mathrm{tr}}
\newcommand{\expTGM}{\exp^L}

\title{Cotangent Bundles as Coadjoint Orbits and Asymptotic Character Formulas}
\author{M. Gjertsen\footnote{NTNU Trondheim, \href{mailto:michael.gjertsen@ntnu.no}{michael.gjertsen@ntnu.no}} \ and A. Schmeding\footnote{NTNU Trondheim, \href{mailto:alexander.schmeding@ntnu.no}{alexander.schmeding@ntnu.no}}}
\date{}

\begin{document}

\maketitle
\begin{abstract}
\noindent We present geometric quantization on cotangent bundles as a special instance of Kirillov's orbit method. To this end, the cotangent bundle is realized as a coadjoint orbit of an infinite-dimensional Lie group for a variety of topologies on the dual of its Lie algebra. The group, central to quantum mechanics on manifolds, is the semidirect product constructed from the diffeomorphism group acting on the abelian group of smooth functions on the manifold. We then develop asymptotic character formulas by employing the tangent groupoid of the manifold and investigating the interplay between its groupoid C$^\ast$-algebra and the infinite-dimensional Lie group.
\end{abstract}

\textbf{MSC2020:} 53D50  
(Primary); 17B08, 
 22E65,  
 22E66, 
 22A22 
 (Secondary) \smallskip

\textbf{Keywords:} geometric quantization, infinite-dimensional Lie group, coadjoint orbit, Connes' tangent groupoid, orbit method, character formula, trace formula
\tableofcontents

\section{Introduction}
The present article presents geometric quantization on cotangent bundles as a special instance of Kirillov's orbit method for infinite-dimensional Lie groups. 
The orbit method is founded on the observation that the coadjoint orbits of a Lie group $G$ are symplectic manifolds, where the symplectic form is canonically determined by the Lie bracket on $\mathfrak g$; a \textit{coadjoint orbit} is an orbit of $G$ in $\mathfrak g^*$, the continuous dual of its Lie algebra $\mathfrak g$, with respect to the dual of the adjoint action—the \textit{coadjoint action}. These symplectic manifolds give rise to irreducible representations of $G$ via an algebraic quantization procedure due to Kirillov \cite{Kirillov}. The orbit method spurred the development of geometric quantization, a geometric procedure which constructs a Hilbert space with a class of self-adjoint operators from a class of smooth functions on a symplectic manifold $S$ \cite{Woo97}. This can be viewed as a generalization of Kirillov's procedure:\ The coadjoint action of $G$ on a coadjoint orbit $S \subset \mathfrak g^*$ is Hamiltonian, with comoment map $\mathfrak g \to C^\infty(S, \R)$, $Z \mapsto \ev_Z$, where $\ev_Z(\phi) = \phi(Z)$. Geometric quantization of the image of this map will produce the derived version of the irreducible representation obtained via Kirillov's method.

The most common symplectic manifolds in physics are cotangent bundles:\ The dynamics of a classical system on a smooth manifold $M$ are described using its cotangent bundle $T^*M$ with the canonical symplectic structure. In this case, the class of functions that are amenable to geometric quantization are those that are affine on each fiber ($T_x^*M$ for $x \in M$). It is well-known that these are the Hamiltonian functions corresponding to an action by the infinite-dimensional Lie group
\begin{align*}
    \GM \coloneqq \Diff_c(M) \ltimes C_c^\infty(M, \R) \qquad \text{on } T^*M.\text{\footnotemark}
\end{align*}
It is therefore natural to ask whether $T^*M$ can be realized as a coadjoint orbit of $\GM$ and if geometric quantization can be recovered via the orbit method. We have not been able to find a discussion of this in the literature, presumably due to the difficulties that stem from the fact that $\GM$ is infinite-dimensional. In this article, we will therefore show that it is indeed possible to recover geometric quantization via the orbit method. Our results from Sections \ref{sec:cotangent_bundle_as_coadjoint_orbit} and \ref{sec:geometric_quantization_result}, in particular Theorems \ref{thm:coadjoint_orbit} and \ref{thm:Kirillovs_quantization}, subsume:
\footnotetext{This discussion assumes that we have chosen the (canonical) vertical polarization on the cotangent bundle. The compact support condition is not necessary for geometric quantization, but for Lie theoretic purposes. See \Cref{sec:geom_quant} for further discussion.}
\begin{mainthm}
    Let $M$ be a connected smooth manifold. Then, the cotangent bundle $T^*M$ is symplectomorphic to a coadjoint orbit of $\GM$ with respect to any locally convex topology on $\gM^*$ (the continuous dual of its Lie algebra $\gM$) between the weak$^*$- and strong topologies.\footnote{We do not attempt to develop a general theory of coadjoint orbits in the infinite-dimensional setting. The point is that $T^*M$ is isomorphic to an orbit of $\GM$ under its coadjoint action on $\gM^*$, that the smooth manifold structure on $T^*M$ is recovered from a variety of locally convex topologies on $\gM^*$ (each of which turns $\gM^*$ into an infinite-dimensional smooth manifold) and that the canonical symplectic form on $T^*M$ is obtained by the Kirillov--Kostant--Souriau prescription applied to this orbit.} Moreover, applying Kirillov's quantization scheme to this coadjoint orbit recovers geometric quantization.
\end{mainthm}

\noindent The group $\GM$ is often considered in relation to quantum mechanics on $M$, even beyond the setting of geometric quantization. Guillemin and Sternberg \cite[Section~29]{GuilleminSternberg} and Isham \cite{Isham84} have considered coadjoint orbits of certain subgroups of $\GM$, while Goldin, Menikoff and Sharp \cite{GoldinEtAl80} have obtained the quantum mechanics of multi-particle systems on $M = \Rd$ by considering the representation theory of a completion of $\GM$. Goldin and Sharp \cite{GoldinSharp25} recently described this completion as ``a universal kinematic group for quantum mechanics''. In \cite{MR2544305}, Beltiţă and Beltiţă realize the cotangent bundle $T^*G$ of a finite-dimensional Lie group $G$ as a coadjoint orbit of groups the form $G \ltimes \mathcal F$, where $\mathcal F \subset C^\infty(G, \R)$ is a potentially infinite dimensional subspace invariant under translation. Modulo the compact support condition, such groups are subgroups of $\mathcal G_M$  in the case where $M = G$. There has also been considerable work on the orbit method for infinite-dimensional Lie groups more broadly \cite{Gay-BalmazFrancois2024TRSD, HallerStefan2024Wnfm, Khesin_Wendt}. 

The character formula is an important facet of the orbit method. Suppose we have a unitary representation $\rho \colon G \to \mathcal U(H)$ of a finite-dimensional Lie group $G$ obtained via the orbit method from a coadjoint orbit $S \subset \mathfrak g^*$, where $\mathcal U(H)$ denotes the unitary operators on the Hilbert space $H$.\footnote{In particular, we are assuming that $S$ is an orbit to which the orbit method applies. For the precise correspondence between coadjoint orbits and representations, see e.g.\ \cite{Kirillov}.} If $H$ is finite-dimensional, the character formula states that, under certain circumstances (which will be discussed in a moment), 
\begin{align}
    \label{eq:fin_dim_character_formula}
    \tr(\rho_{\exp Z}) = J(Z)^{-1/2}\int_S e^{-2 \pi i \langle \phi, Z \rangle} dV(\phi) \quad \text{for } Z \text{ in a neighborhood of } 0 \in \mathfrak g,\footnotemark
\end{align}
where $\langle \cdot, \cdot \rangle\colon \mathfrak g^* \times \mathfrak g \to \R$ is the natural pairing, $dV \coloneqq \omega^n/n!$ ($2n = \dim S$) is the volume form associated to the canonical (Kirillov--Kostant--Souriau) symplectic form $\omega$ on $S$, $\ad_X \colon \mathfrak g \to \mathfrak g$ is the adjoint action of $X \in \mathfrak g$ and $J$ is the \textit{Duflo factor} defined by \footnotetext{We induce from the character $\exp(Z) \mapsto e^{-2 \pi i \langle \phi, Z \rangle}$ instead of $\exp(Z) \mapsto e^{2 \pi i \langle \phi, Z \rangle}$ (as in \cite{Kirillov}), as this leads to consistency with our conventions for geometric quantization and the tangent groupoid.}
\begin{align*}
    J(Z) \coloneqq \det \left( \frac{\sinh(\ad_{Z/2})}{\ad_{Z/2}} \right) \quad \text{and} \quad J(0) \coloneqq 1.
\end{align*}
In the case of a not necessarily finite-dimensional Hilbert space $H$, there is an analogous formula for the \textit{distributional character}, defined as $\Theta_\rho \colon C_c^\infty(G) \to \C$, $\Theta_\rho(f) = \tr\, \rho(f)$, where $\rho \colon C_c^\infty(G) \to \mathcal B(H)$ is the integrated representation and $\mathcal B(H)$ denotes the bounded operators on $H$. The formula for the distributional character states that, under certain circumstances,
\begin{align} \label{eq:dist_character_formula}
    \Theta_\rho( f ) = \tr \,\rho(f) = \int_{S} \mathcal F \bigl(J^{1/2} \cdot (f \circ \exp)\bigr) \,dV \quad \text{for } f \in C_c^\infty(G), 
\end{align}
where $\mathcal F\colon L^2(\mathfrak{g}) \rightarrow L^2(\mathfrak{g}^\ast)$ is the Fourier transform (defined using the normalized Lebesgue measure on $\mathfrak g$ determined by the Haar measure on $G$). The character formula \eqref{eq:fin_dim_character_formula} is valid for simply connected compact Lie groups \cite[Theorem~5.9]{Kirillov}, where all irreducible representations are finite-dimensional and the orbit method yields a bijection between regular integral coadjoint orbits and irreducible representations \cite[Chapter~5]{Kirillov}. For $f$ supported in a neighborhood on which the exponential map is a diffeomorphism, \eqref{eq:dist_character_formula} is a consequence of \eqref{eq:fin_dim_character_formula}. The distributional formula \eqref{eq:dist_character_formula} also applies to all simply connected nilpotent Lie groups, for which the exponential map is a diffeomorphism, $J \equiv 1$, and the orbit method yields a bijection between coadjoint orbits and irreducible representations \cite[Chapter~3]{Kirillov}. Other situations in which the character formulas apply are discussed in \cite{Kirillov}, see e.g.\ \cite[Chapter~4]{Kirillov}.

If $G$ is an infinite-dimensional Lie group, then neither the Haar measure nor the Lebesgue measure is available, hence there is no integrated representation or Fourier transform. This leaves us unable to even define the distributional character. A slight change of perspective, however, enables us to formulate an analogue of the character formula for the coadjoint orbit $T^*M \subset \gM^*$. From the point of view of quantization \cite{Landsman_Topics}, the character formula in the finite-dimensional setting can be expressed as follows:\ We begin with the function $a \coloneqq \mathcal F(J^{1/2} \cdot (f \circ  \exp))$ on $\mathfrak g^*$, whose quantization is $\mathcal Q(a) \coloneqq \rho(f)$, and the character formula becomes
\begin{align*}
    \tr \, \mathcal Q(a) = \int_S a \,d V.
\end{align*}
In the case of $\GM$, there are various quantization procedures taking functions on the coadjoint orbit $T^*M \subset \gM^*$ to operators on $L^2(M)$, in a manner that extends geometric quantization, and they satisfy variants of the trace formula above \cite{Landsman_Topics, CariñenaEtAl} (see also \Cref{rem:character_formula} and \Cref{cor:all_orders}). Using the tangent groupoid $\TGM$ of $M$, we can consider the more flexible (and physical) notion of \textit{dequantization}. Namely, we can make sense of the statement that a family of operators $(T_h)_{h \in (0, 1]}$ on $L^2(M)$ converges to a function $a \colon T^*M \to \C$ as $h \to 0$, and any sensible quantization scheme $a \mapsto (\mathcal Q_h a)_{h \in (0, 1]}$ should then have the property that $\mathcal Q_h a \to a$ as $h \to 0$. Here, $h \in [0, 1]$ is referred to as a \textit{semiclassical parameter} and $h \to 0$ as the \textit{classical limit}. We refer to $a$ as the \textit{asymptotic symbol} of the family $(T_h)_{h \in (0, 1]}$. The simplest condition for the convergence $T_h \to a$ is that the tuple $(a, T_h)_{h \in (0,1]}$ corresponds to a smooth and compactly supported function on the tangent groupoid, which we write as $(a, T_h) \in C_c^\infty(\TGM)$ (see \Cref{sec:prelim_tangent_groupoid} for details and the discussion preceding \Cref{thm:double_centralizer} for a more flexible condition). With all of this in mind, we refer to the following result (\Cref{prop:easy_trace_formula}) as an asymptotic character formula for the coadjoint orbit $T^*M \subset \gM^*$. See \Cref{cor:all_orders} for an expansion to all orders in $h$.
\begin{mainthm} \label{mainthm2}
    For any $(a, T_h) \in C_c^\infty(\TGM)$, 
    \begin{align*}
        \lim_{h \to 0} h^d \, \tr \,T_h = \int_{T^*M} a \,dV_L,
    \end{align*}
    where $V_L$ denotes the Liouville measure on $T^*M$ and $d \coloneqq \dim M$.
\end{mainthm}
\noindent Moreover, letting $\rho^h$ denote the representations of $\GM$ obtained via Kirillov's method, we show that for $Z \in \gM$, the function $\eta \mapsto e^{-2 \pi i \langle \eta, Z \rangle}$ on $T^*M \subset \gM^*$ is the classical limit of the family $(\rho^h_{\exp h Z})_{h \in (0, 1]}$ (\Cref{thm:double_centralizer}). This gives:
\begin{mainthm}
    For any $(a, T_h) \in C_c^\infty(\TGM)$ and $Z \in \gM$, 
    \begin{align*}
    \lim_{h \to 0} h^d \tr ( \rho^h_{\exp h Z} \,T_h) = \int_{T^*M} a(\eta) e^{-2\pi i \langle \eta, Z \rangle} \,dV_L(\eta).
\end{align*}
\end{mainthm}

\noindent This is \Cref{cor:character_formula}. It may be thought of as a regularized asymptotic version of the finite-dimensional character formula \eqref{eq:fin_dim_character_formula}; see \Cref{rem:relation_to_Kirillov}.

The key to these results is \Cref{lem:diffeo_extension_to_TGM}, an extension lemma for diffeomorphisms of the tangent groupoid. This allows us to establish multiple novel results regarding the structure of the tangent groupoid and its groupoid C$^*$-algebra, in particular as it relates to the natural action of $\GM$ and geometric quantization. In addition to the trace formulas already discussed, we generalize a result regarding Weyl quantization on the tangent groupoid, originally due to Landsman \cite{Landsman1993}. Our generalization, \Cref{prop:translation_of_ops}, shows that the coadjoint action of $\GM$ on $T^*M \subset \gM^*$ is the classical limit of conjugation by $\rho^h$ on $\mathcal B(L^2(M))$. All of these results regarding the tangent groupoid can be thought of as general statements about dequantization. As such, they are also valid for any quantization scheme that is compatible with the structure of the tangent groupoid, including the Weyl quantization developed by Landsman.\footnote{Our results are of the following form:\ For a family $(T_h)_{h \in (0, 1]}$ with asymptotic symbol $a$, there is a certain relation between $T_h$ and $a$. A quantization scheme $a \mapsto (\mathcal Q_ha)_{h \in (0, 1]}$ that is compatible with the structure of the tangent groupoid has the property that $a$ is the asymptotic symbol of $(\mathcal Q_ha)_{h \in (0, 1]}$. Thus, our results are valid, in particular, for $T_h = \mathcal Q_ha$.} We construct Weyl quantization and other related quantization schemes in \Cref{app:quantization} to help illustrate this point and showcase the utility of the tangent groupoid for quantization and dequantization. Besides \Cref{app:quantization} and \Cref{cor:all_orders}, the only quantization schemes we consider are geometric quantization and Kirillov's algebraic method, which we show to coincide in this setting. 

\paragraph{Contribution and Relation to Other Work} We now summarize the results of this article and their relation to the literature. 

\begin{enumerate}[label=(\roman*)]
    \item \textit{$T^*M$ is a coadjoint orbit of $\GM$} (\Cref{thm:coadjoint_orbit}). The non-trivial and novel aspect of this result is its topological content:\ The well-known moment map from \Cref{prop:moment_map} \cite{Landsman_Topics, GuilleminSternberg} is a smooth (symplectomorphic) embedding of $T^*M$ for every locally convex topology on $\gM^*$ between the weak$^*$- and strong topologies. This result makes rigorous and further develops purely formal Lie theoretic discussions in \cite{Landsman_Topics, GuilleminSternberg, Isham84, GoldinEtAl80}. The group $\GM$, along with its subgroups and completions, are central to the study of quantum mechanics on Riemannian manifolds \cite{GoldinSharp25, Landsman_Topics, GuilleminSternberg, MR2544305, GoldinEtAl80}. In \Cref{sec:conclusions}, we discuss future directions for the study of $\GM$ and its subgroups suggested by our results. Other coadjoint orbits of diffeomorphism groups are studied in \cite{HallerStefan2024Wnfm, Khesin_Wendt, AliGoldin91}.
    \item \textit{Geometric quantization on $T^*M$ is Kirillov's algebraic orbit method} (\Cref{thm:Kirillovs_quantization}). The representations $\rho^h$ are widely studied, and their relation to geometric quantization is well-known \cite{Landsman_Topics, GoldinSharp25}. Our contribution is the construction of these representations using the algebraic quantization procedure of Kirillov (which inspired geometric quantization to begin with) in the setting of infinite-dimensional Lie theory. This is a broad emerging topic lacking an overarching theory \cite{MR2544305, Gay-BalmazFrancois2024TRSD, HallerStefan2024Wnfm, Khesin_Wendt, Kosyak}; it is also one of the future directions for the orbit method highlighted by Kirillov \cite[Section~6.3.1]{Kirillov}. Our result provides a natural class of examples directly related to quantum mechanics on Riemannian manifolds. It also suggests further classes of examples, which we discuss in \Cref{sec:conclusions}. 
    \item \textit{Character formulas} (\Cref{prop:easy_trace_formula}, \Cref{cor:character_formula} and  \Cref{cor:all_orders}). Trace formulas analogous to that of \Cref{mainthm2} (\Cref{prop:easy_trace_formula}) are commonplace:\   \cite{CariñenaEtAl} contains an exact (non-asymptotic) formula for Weyl quantization obtained using the tangent groupoid (generalizing the well-known case of $M = \R^d$); asymptotic variants occur in the context of Berezin--Toeplitz quantization of Kähler manifolds \cite{Berman06, Schlichenmaier00, BaMaS94}. Our contribution is the general quantization-scheme-independent formulation, which shows that all quantization schemes compatible with the tangent groupoid lead to such asymptotic trace formulas, as well as their representation theoretic interpretation in the context of the orbit method (\Cref{cor:character_formula} and \Cref{rem:relation_to_Kirillov}). Using similar methods, we also obtain higher order expansions with error estimates, as well as exact formulas for a natural class of quantization schemes (\Cref{cor:all_orders} and \Cref{app:quantization}).
    \item \textit{The structure of $C^*(\TGM)$ and the classical limit} (\Cref{lem:diffeo_extension_to_TGM}, \Cref{thm:double_centralizer} and \Cref{prop:translation_of_ops}). The C$^*$-algebra $C^*(\TGM)$ is an important object in non-commutative geometry, index theory and the approach to quantization initiated by Landsman \cite{Connes_NG, HigsonNigel2019SatT, Landsman1993, Landsman_Foundations, CariñenaEtAl}. \Cref{lem:diffeo_extension_to_TGM}, an extension theorem for diffeomorphisms of the tangent groupoid $\TGM$, is, to the best of our knowledge, novel. It readily leads to structural results for $C^*(\TGM)$ related to the classical limit of quantum mechanics on $M$:\ \Cref{prop:translation_of_ops} identifies the natural action of $\GM$ on $T^*M$ as the classical limit of conjugation by $\rho^h$, generalizing Landsman's result for Weyl quantization \cite[Theorem 5]{Landsman1993} by constructing a natural fiber-preserving action of $\GM$ on $C^*(\TGM)$; \Cref{thm:double_centralizer} identifies the classical limit of $\rho^h$ by constructing elements of the multiplier algebra of $C^*(\TGM)$.
\end{enumerate}

\paragraph{Structure of the Paper}
We end the introduction by establishing some standard notation which is used throughout the article. In \Cref{sec:geom_quant}, we show how geometric quantization on cotangent bundles can be understood as an instance of Kirillov's orbit method. The information needed on infinite-dimensional spaces and their geometry is collected in \Cref{app:mfd_smooth_maps}. While the material in the appendix is needed to follow all the proofs in detail, it can be skipped on first reading or taken as a reference once the constructions are needed. In \Cref{sect:asymptotic_char}, we establish our second main result, the asymptotic character formula(s). This is embedded in a larger exploration of the structure of Connes' tangent groupoid and its groupoid C$^*$-algebra. The tangent groupoid is constructed from scratch in \Cref{sec:prelim_tangent_groupoid}, and our results are located in \Cref{sec:asymptotic_character_formula}. The main text concludes with \Cref{sec:conclusions}, a discussion of future research directions suggested by our results. \Cref{app:proofs} is a collection of proofs of results that are well-known but for which we could not locate references with sufficiently detailed and elementary proofs. The results themselves are reproduced in the text, making \Cref{app:proofs} optional. We recommend that the interested reader visits this appendix as the results appear in the main body of the text. Finally, \Cref{app:quantization} is a construction of quantization procedures on the tangent groupoid. It is entirely optional and only intended to illustrate some quantization procedures to which our general results regarding dequantization can be applied (see e.g.\ \Cref{cor:all_orders}, which is the only place in the main text where these quantization schemes are invoked).

\paragraph{Acknowledgements}

We would like to thank Franz Luef for many valuable conversations and for the impetus to explore these structures in the first place. We would also like to thank the anonymous referees for many useful suggestions which helped improve our presentation substantially, clarified the precise relation between \Cref{mainthm2} and \eqref{eq:fin_dim_character_formula} (\Cref{rem:relation_to_Kirillov}) and motivated the inclusion of \Cref{cor:all_orders}.\footnote{\Cref{cor:all_orders} is the only result in this paper obtained with the aid of artificial intelligence (see the statement regarding AI following \Cref{app:quantization}).} The potential generalization to symplectic groupoids discussed in \Cref{sec:conclusions} was also suggested by one of the referees.

\begin{setup}[Conventions and Notations] \label{setup:notation}
    We write $\N = \{1,2,3,\ldots\}$ for the natural numbers and $\N_0 \coloneq \N \cup \{0\}$. By default, every finite-dimensional manifold, whose dimension will be denoted by $d$, is assumed to be $\sigma$-compact\footnote{Recall that a topological space is $\sigma$-compact if it is a countable union of compact subsets.} and connected. This implies in particular that it is paracompact and admits a Riemannian metric $g$. In order to avoid the unnecessary assumption of orientability, we work with densities instead of volume forms for integration on manifolds.
 
    We write $C^\infty (M,N)$ for the set of smooth functions between two smooth manifolds, and if $E$ is a vector space, we write $C^\infty_c (M,E)$ for the space of compactly supported functions. This is shortened to $C^\infty (M)\coloneq C^\infty (M,\mathbb{C})$ (and similarly for the compactly supported functions). We write $\X(M)$ for the space of vector fields on $M$ (all vectors fields are assumed to be smooth) and $X_x$ for the value of $X \in \X(M)$ at $x \in M$, while a generic vector in $T_xM$ is denoted $v_x$. Similarly, a generic element of $T_x^*M$ is denoted $\eta_x$, as is the evaluation of a (smooth) one-form at $x$. 
    
    The continuous dual of a topological vector space $V$, the vector space of all continuous linear maps $V \to \C$, is denoted by $V^*$. 
\end{setup}

\begin{setup}[Coordinates]\label{setup:coordinates}Let $M$ be a finite-dimensional manifold, $U \subset M$ a chart domain and $q \colon U \to \Rd$ a chart. We write $(q, v) \colon TU \to \R^{2d}$ and $(q,p) \colon T^*U \to \R^{2d}$ for the associated \textit{tangent} and \textit{cotangent coordinates}. These are defined, for $x \in U$, by 
\begin{align*}
    T_xM \ni \sum\nolimits_{j=1}^d v^j \frac{\partial}{\partial q^j}\Bigl|_x &\mapsto \bigl(q^1(x), \ldots,  q^d(x), v^1, \ldots, v^d \bigr) \in \R^{2d} \\
   \text{and} \quad T_x^*M \ni \sum\nolimits_{j=1}^d p_j dq^j_{x} &\mapsto \bigl( q^1(x), \ldots, q^d(x), p_1, \ldots, p_d \bigr) \in \R^{2d}.
\end{align*}
If $g$ is a Riemannian metric on $M$, we let $g_q \colon U \to \mathrm{GL}(d, \R)$ be the induced matrix representation of the metric, i.e.\ $g_q(x)_{ij} \coloneqq g(\frac{\partial}{\partial q^i}|_x, \frac{\partial}{\partial q^j}|_x)$. For the coordinate representation $q(U) \to  \mathrm{GL}(d, \R)$ of $g_q$, we write $g(q)$.  Recall the \textit{Riemannian density} $dV_g$ on $M$ \cite[Proposition 16.45]{Lee13}, given in coordinates by
\begin{align*}
    d\vol(q) = \sqrt{|\det g(q)|} \, d^d q \coloneqq \sqrt{|\det g(q)|} \cdot | dq^1 \wedge \cdots \wedge dq^d |.
\end{align*}
We refer to the corresponding Borel measure $V_g(A) \coloneqq \int_A d\vol$ as the \textit{Riemannian measure} and write $L^2(M) \coloneqq L^2(M, V_g)$, where the functions are $\C$-valued. 
\end{setup}

\begin{setup}[$\R$ or $\C$?] \label{setup:R_or_C}
    The interplay between $\R$-valued and $\C$-valued function spaces can be confusing. Spaces of real-valued functions always include $\R$ in the notation, such as $C^\infty(M, \R)$, while $\C$ is usually implicit, e.g.\ $C^\infty(M) \coloneqq C^\infty(M, \C)$. To aid the reader, we offer the following guide that can be returned to if confusion arises.
    \begin{itemize}
        \item Function spaces pertaining directly to the manifold structure of $M$ (such as coordinates and forms), the group $\GM$ (such as its Lie algebra $\gM$ \eqref{eq:Lie_algebra} and the continuous dual $\gM^*$  \eqref{eq:Lie_algebra_dual}), or the Hamiltonians $H_{(X, f)}$ \eqref{eq:def_Hamiltonians}, are $\R$-valued.
        \item The space $L^2(M)$, its subspaces, as well as all function spaces pertaining to the tangent groupoid $\TGM$ (function spaces on $T^*M$ and $M \times M$; \Cref{sec:prelim_tangent_groupoid}), are $\C$-valued.
    \end{itemize}
\end{setup}

\section{Geometric Quantization on Cotangent Bundles}\label{sec:geom_quant}

This section begins with a review of geometric quantization on a cotangent bundle $T^*M$ and its connection with a family of representations $\rho^h \colon \GM \to \mathcal U(L^2(M))$. This involves the moment map $\mu \colon T^*M \to \gM^*$ for a particular Hamiltonian action $\alpha^0 \colon \GM \to \Diff(T^*M)$. This moment map is shown to be a smooth embedding of a coadjoint orbit in \Cref{sec:cotangent_bundle_as_coadjoint_orbit}. In \Cref{sec:geometric_quantization_result}, we show how to obtain geometric quantization on $T^*M$ via Kirillov's quantization procedure. 

Geometric quantization on a symplectic manifold $S$ involves multiple choices. Most crucial is the choice of a \textit{polarization}, which is a foliation of $S$ whose leaves are Lagrangian submanifolds. In the case of a cotangent bundle $S = T^*M$, the canonical choice is the \textit{vertical polarization}: its leaves are the fibers of the bundle and the space of leaves is identified with $M$. The (straightforwardly) quantizable functions on the cotangent bundle (with the vertical polarization) are those that are affine on each fiber, each of which can be written as
\begin{align} \label{eq:def_Hamiltonians}
    H_{(X, f)}(\eta_x) \coloneqq \eta_x(X_x) + f(x) \quad \text{for } X \in \X(M) \text{ and } f \in C^\infty(M, \R),
\end{align}
where $\eta_x \in T^*M$. Choosing a Riemannian metric on $M$, we can identify the quantum Hilbert space with $L^2(M)$. For each $h \in (0, 1]$, the quantization $\mathcal Q_h H_{(X, f)}$ of $H_{(X, f)}$ will then act on $\psi$ in its domain $D_{(X, f)} \subset L^2(M)$ as
\begin{align} \label{eq:def_quantized_Hamiltonians}
    \mathcal Q_h H_{(X, f)} \psi = f \cdot \psi - \frac{ih}{2\pi} \Bigl( X \psi + \frac12(\mathrm{div}\,X) \cdot \psi \Bigr),
\end{align}
where $\mathrm{div}\,X$ is the divergence of $X$ and $\cdot$ denotes pointwise multiplication. See Hall \cite[Example~23.48]{Hall13} for details.\footnote{The expression for $\mathcal Q_h H_{(X, f)}$ in \cite{Hall13} has different signs, due to different conventions. See the lecture notes by Schottenloher \cite{Schottenloher} for an exposition using our conventions.} In order to obtain a common domain, say $C_c^\infty(M) \subset L^2(M)$, we impose compact support conditions:\ $X \in \X_c(M)$ and $f \in C_c^\infty(M, \R)$. We will now see that these operators arise as the derived version of a representation $\rho^h$ of an infinite-dimensional Lie group $\GM$, and that the functions $H_{(X, f)}$ are the Hamiltonians for an action of $\GM$ on $T^*M$. All of this, except for a discussion of the manifold structure on $\GM$, can be found in Landsman \cite{Landsman_Topics}.

Consider the group $\DiffcM$ of compactly supported diffeomorphisms of $M$ as an infinite-dimensional Lie group \cite{Neeb06}. We recall essentials on the manifold structure and function space topologies in \Cref{app:mfd_smooth_maps}. On the locally convex space $C_c^\infty(M, \R)$ of compactly supported smooth functions on $M$, we have a natural Lie group action by $\DiffcM$, namely 
\begin{align}
   \gamma(\phi,f) \coloneqq f \circ \phi^{-1} \quad \text{for } \phi \in \DiffcM \text{ and } f \in C_c^\infty(M, \R).
\end{align} 
As in the finite-dimensional setting, we can consider the semidirect product
\begin{align*}
    \GM \coloneqq \DiffcM \ltimes_\gamma \CcMR
\end{align*}
of Lie groups by viewing the compactly supported functions as an abelian Lie group. The product and inverse in $\GM$ are given by
\begin{align}
    \label{eq:prod_and_inv}
    (\phi, f) (\theta, g) = (\phi \circ \theta, g \circ \phi^{-1} + f) \quad \text{and} \quad (\phi, f)^{-1} = (\phi^{-1}, - f \circ  \phi).
\end{align}
The Lie algebra of the semidirect product $\GM$ is the semidirect sum
\begin{align} \label{eq:Lie_algebra}
    \gM = \XcM \ltimes_{\mathrm{d}\gamma} \CcMR,
\end{align}
where $\XcM$ is the space of compactly supported smooth vector fields on $M$ with the negative of the Lie bracket of vector fields. The Lie bracket on $\gM$ is 
\begin{align}
    \label{eq:Liebracket}
    \bigl[ (X,f), (Y, g) \bigr] = \bigl( -[X, Y], - Xg + Yf \bigr),
\end{align}
(see \Cref{gm:Liebracket}) and the exponential map $\exp \colon \gM \to \GM$ is given by
\begin{align}
    \label{eq:exp_map}
    \exp(X, f) = \biggl( \Fl^X_1,  \int_0^1 f \circ \Fl^{-X}_{t} \, dt \biggr),
\end{align}
where $\Fl^X \colon [0,1] \rightarrow \Diff_c(M), t \mapsto \text{Fl}^X_t$ is the flow of $X$, i.e.\ $\Fl^X_0 = \id_M$ and $\left.\ddt\right|_{t=s}\Fl^X_t = X \circ \Fl_s^X$ for all $s \in [0, 1]$, and the integral on $C_c^\infty(M, \R)$ is to be interpreted in the sense of weak integrals. See \cite[Section~38]{KM97} for details.

There is a natural action of $\GM$ on the cotangent bundle $T^*M$, which we denote by $\alpha^0 \colon \GM \to \Diff(T^*M)$. For $(\phi, f) \in \GM$, it is given as
\begin{align}
    \label{eq:alpha_0}
    \alpha^0_{(\phi, f)}(\eta_x) = (\phi^{-1})^* \eta_x - df_{\phi(x)} \quad \text{for } x \in M \text{ and } \eta_x \in T_x^*M,
\end{align}
where $\alpha^0_{(\phi, f)} \coloneqq \alpha^0(\phi, f) \in \Diff(T^*M)$ and $(\phi^{-1})^*\eta_x \coloneqq \eta_x \circ T_{\phi(x)}\phi^{-1}$ is the pullback of $\eta_x$ by $\phi^{-1}$. The proof of the next proposition is deferred to \Cref{app:proofs}, \Cref{cor:app_comoment_map}. For the notion of moment maps and comoment maps, we refer to \cite[Chapter~22]{SilvaAna}.\footnote{The equivariance condition on (co)moment maps in \cite{SilvaAna} is not always imposed. In particular, we do not show equivariance in our proof of \Cref{prop:moment_map}, but it trivially follows from our main result, \Cref{thm:coadjoint_orbit}.} Note the appearance of the Hamiltonians $H_{(X, f)}$ from  \eqref{eq:def_Hamiltonians}. 

\begin{prop}[{The Moment Map of $\alpha^0$,  \cite[Section~29]{GuilleminSternberg} and \cite[Section~4.7.3]{Isham84}}]
    \label{prop:moment_map}
    The action $\alpha^0 \colon \GM \to \Diff(T^*M)$ given by \eqref{eq:alpha_0} is Hamiltonian, with comoment map 
    \begin{align*}
        H \colon \gM &\to C^\infty(T^*M, \R) \\
        (X, f) &\mapsto H_{(X, f)} \colon \eta_x \mapsto \eta_x(X_x) + f(x).
    \end{align*}
    Consequently, the moment map is 
    \begin{align*}
        \cO \colon T^*M &\to \gM^* \\ 
        \eta_x &\mapsto (\eta_x, \delta_x) \colon (X, f) \mapsto \eta_x(X_x) + f(x),
    \end{align*}
    where $\gM^* = \X_c(M)^\ast \times \CcMR^\ast$ is the continuous dual of $\gM$, the map $\delta_x \colon C_c^\infty(M, \R) \to \R$ is evaluation at $x$, i.e.\ $\delta_x(f) = f(x)$, and $\eta_x \in T_x^*M$ is identified with the map $\X_c(M) \to \R, X \mapsto \eta_x(X_x)$ (i.e.\ we evaluate the vector field at $x$ and apply $\eta_x$).
\end{prop} 

This entails that $H_{(X, f)}$ is the Hamiltonian generating the one-parameter flow $t \mapsto \alpha^0_{\exp(tX, tf)}$ on $T^*M$. The moment map is obtained by evaluation of the comoment map:\ $\mu(\eta_x) \in \gM^*$ maps $(X, f)$ to $H_{(X, f)}(\eta_x)$, which happens to coincide with $(\eta_x, \delta_x) \in  \X_c(M)^\ast \times \CcMR^\ast$.

Given a Riemannian metric $g$ on $M$, there is also an $h$-dependent family of unitary representations of $\GM$ on $L^2(M)$ ($\C$-valued; see \Cref{setup:R_or_C}), with $h \in (0, 1]$, which we denote by $\rho^h \colon \GM \to \mathcal U(L^2(M))$. These are defined, for $(\phi, f) \in \GM$ and $\psi \in L^2(M)$, by
\begin{align}
    \label{eq:def_of_rho}
    \rho^h_{(\phi, f)}\psi(x) = e^{-2 \pi i f(x) /h} \sqrt{\frac{d (\phi^{-1})^* \vol}{d\vol}}(x) \psi \bigl(\phi^{-1}(x) \bigr),
\end{align}
where $(\phi^{-1})^*\vol$ is the pullback of the Riemannian measure $\vol$ under $\phi^{-1}$, i.e.\ $(\phi^{-1})^*\vol(A) = \vol(\phi^{-1}(A))$ for all measurable sets $A \subset M$. It will be convenient to have an explicit expression for the Radon--Nikodym derivative. For the proof, see \Cref{app:proofs}, \Cref{app_Radon_Nikodym}.

\begin{lemma}[Radon--Nikodym Derivative in Coordinates]
    \label{lem:Radon_Nikodym_formula}
    Let $\phi \in \Diff_c(M)$ and $q \colon U \to \Rd$ be coordinates. Then, 
    \begin{align*}
        \frac{d (\phi^{-1})^* \vol}{d\vol}(x) = \Biggl| \frac{\det g_{q \circ \phi}(\phi^{-1}(x))}{\det g_q(x)} \Biggr|^{\frac12}  =  \Biggl| \frac{\det g\bigl(T\phi^{-1} \frac{\partial}{\partial q^i}\big|_x, T\phi^{-1} \frac{\partial}{\partial q^j}\big|_x \bigr)}{\det g \bigl(\frac{\partial}{\partial q^i}\big|_x, \frac{\partial}{\partial q^j}\big|_x \bigr)} \Biggr|^{\frac12}
    \end{align*}
    for all $x \in U$. In particular, this is a smooth function on $M$.
\end{lemma}

\noindent The derived representation $d\rho^h$ of $\gM$ corresponds to geometric quantization. Our convention is that $\rho^h_{\exp(X, f)} = \exp(d\rho^h_{(X, f)})$, where $d\rho^h_{(X, f)}$ is the skew-adjoint generator (given by Stone's theorem) of the one-parameter group $t \mapsto \rho^h_{\exp(tX, tf)}$, and $\exp(d\rho^h_{(X, f)})$ is defined using the spectral calculus \cite[Theorem~10.15]{Hall13}. For the proof, see \Cref{app:proofs}, \Cref{app_derived_rep}.

\begin{prop}
    \label{prop:derived_rep}
    For $(X, f) \in \gM$ and $\psi \in C_c^\infty(M)$,
    \begin{align*}
        d\rho^h_{(X, f)}\psi  = -\frac{2\pi i }{h}  f \cdot \psi  - \biggl(X\psi + \frac12 (\divv X) \cdot \psi \biggr) = -\frac{2\pi  i}{h} \mathcal Q_h H_{(X,f)},
    \end{align*}
    where $\mathcal Q_h H_{(X, f)}$ is defined by \eqref{eq:def_quantized_Hamiltonians}.
\end{prop}

\subsection{Cotangent Bundles as Coadjoint Orbits}
\label{sec:cotangent_bundle_as_coadjoint_orbit}
In this section, we show that the cotangent bundle $T^*M$ is a coadjoint orbit of the group $\GM$. This does not involve a choice of metric on $M$. 

\begin{lemma}\label{lem:adjoint_action}
   The adjoint action $\Ad_{(\phi, f)} (X, g)$ of $\GM$ on $\gM$ is given by
    \begin{align*}
     \bigl( \Ad_\phi (X), \gamma(\phi, g) -\mathrm{d}\gamma(\Ad_\phi (X), f)  \bigr) = \bigl(  \Ad_\phi (X), g \circ\phi^{-1} + X(f \circ \phi) \circ \phi^{-1} \bigr),
    \end{align*}
    where $\Ad_\phi (X)=T\phi \circ X \circ \phi^{-1}$ is the adjoint action of the group $\Diff_c(M)$.
\end{lemma}
\begin{proof}
    Write $c_g (h)=ghg^{-1}$ for the conjugation action of a group on itself. Then, for all $(\phi, f), (\theta, g) \in \GM$, \eqref{eq:prod_and_inv} yields
    \begin{align*}
        c_{(\phi, f)} (\theta, g)  = \bigl(c_\phi (\theta) , g \circ \phi^{-1} + f - f \circ c_\phi (\theta^{-1}) \bigr).
    \end{align*}
    Let now $X \in \XcM$. Using the smooth curve $[0, 1]\ni t \mapsto (\Fl^X_t, tg) \in \GM$ to compute the derivative of the conjugation, we obtain 
    \begin{align*}
    \Ad_{(\phi,f)}(X,g) &= \left(\Ad_\phi (X) , \frac{d}{dt}\Big|_{t=0} (tg)\circ \phi^{-1} + f - f \circ c_\phi \bigl( (\Fl^X_t)^{-1} \bigr)\right)\\
    &= \left(\Ad_\phi (X), g \circ \phi^{-1} - \frac{d}{dt}\Big|_{t=0} f \circ c_\phi \bigl( (\Fl^X_t)^{-1} \bigr)\right).
    \end{align*}
    The tangent map of $L_f (\theta):=f\circ \theta$ is $L_{Tf}(\xi)=Tf \circ \xi $ (as a partial map of the composition \cite[Corollary 11.6]{Mic80}). We exploit this and use that $c_\phi$ is conjugation in the Lie group $\Diff_c(M)$ and the flow $\Fl_t^X$ is the exponential map of this Lie group. Hence, the usual formulae (e.g.\ \cite[Chapter 38]{KM97}) for the Lie group $\Diff_c(M)$ now imply that $\Ad_{(\phi,f)}(X,g) = (\Ad_\phi (X),g\circ \phi^{-1} + df \circ \Ad_{\phi} (X))$ (where $df$ denotes the $1$-form obtained by derivating $f$), whence the easily verified identity $df \circ \Ad_\phi(X) = X(f \circ \phi) \circ \phi^{-1}$ together with \eqref{GM:derived_action} yields the result.
\end{proof}

Consider now the continuous dual of $\gM$, the vector space
\begin{align} \label{eq:Lie_algebra_dual}
    \gM^\ast = \X_c(M)^\ast \times \CcMR^\ast.
\end{align}
We will denote by $\langle \cdot, \cdot \rangle\colon \gM^* \times \gM \to \R$ the natural evaluation pairing. There is a whole range of natural locally convex topologies on $\gM^*$. The weakest of these is the \textit{weak$^*$-topology}, defined by pointwise convergence, and the strongest is the \textit{strong topology}, defined by uniform convergence on bounded sets. As the resulting locally convex spaces $\gM^*$ are non-normable, differentiability of a map into $\gM^*$ is defined in the sense of Bastiani calculus, i.e.\ all directional derivatives exist and combine to form continuous maps, see \cite{Sch23}.
Recall the moment map
\begin{align}
    \label{eq:coadjoint_orbit_inclusion}
    \cO \colon  T^*M \to \gM^* , \quad
    \eta_x &\mapsto (\eta_x, \delta_x)
\end{align}
from \Cref{prop:moment_map}, where $\eta_x \in T_x^*M$ and
\begin{align*}
    \big\langle (\eta_x, \delta_x), (X, f) \bigr\rangle = \eta_x(X_x) + f(x) \quad \text{for all } (X, f) \in \gM.
\end{align*}
We now show that this is the inclusion of a coadjoint orbit.
\begin{theorem}
    \label{thm:coadjoint_orbit}
    Equip $\gM^*$ with a locally convex topology between the weak$^*$- and the strong topology. Then, the moment map $\cO \colon T^\ast M \rightarrow \gM^\ast$ \eqref{eq:coadjoint_orbit_inclusion} is a smooth embedding\footnote{Recall that a smooth embedding is a smooth immersion which is also a topological embedding.} and a symplectomorphism onto a coadjoint orbit of $\GM$. Moreover, the action $\alpha^0$ \eqref{eq:alpha_0} of $\GM$ on $T^*M$ corresponds to the coadjoint action of $\GM$.
\end{theorem}

\begin{proof}
    The map \eqref{eq:coadjoint_orbit_inclusion} is clearly injective. Using \Cref{lem:adjoint_action}, we find that
    \begin{align*}
        &\blangle \Ad^*_{(\phi, f)} (\eta_x, \delta_x), (X, g) \brangle  =\blangle (\eta_x, \delta_x), \Ad_{(\phi^{-1}, -f \circ \phi)} (X, g) \brangle \\
         &= \blangle (\eta_x, \delta_x), \bigl(\Ad_{\phi^{-1}}( X), g \circ \phi - (Xf) \circ \phi \bigr) \brangle
         \\ &=  \eta_x \bigl( \Ad_{\phi^{-1}}(X)_x \bigr) + g \circ \phi(x) - X_{\phi(x)}f  \\ &= \bigl( (\phi^{-1})^* \eta_x - df_{\phi(x)} \bigr) (X_{\phi(x)}) + \delta_{\phi(x)}g
         = \blangle \bigl( (\phi^{-1})^* \eta_x - df_{\phi(x)}, \delta_{\phi(x)} \bigr), (X, g) \brangle,
    \end{align*}
    and hence $\Ad^*_{(\phi, f)}(\eta_x, \delta_x) = ( (\phi^{-1})^* \eta_x - df_{\phi(x)}, \delta_{\phi(x)} )$, which shows that $\alpha^0$ corresponds to the coadjoint action. Now, $\DiffcM$ acts transitively on the connected manifold $M$ \cite{MV}, and the differentials of smooth and compactly supported functions span the cotangent space at every point. Thus, for fixed $y \in M$ and $\beta_y \in T_y^*M$, the coadjoint orbit $\mathcal O_{(\beta_y, \delta_y)}$ coincides with the image of \eqref{eq:coadjoint_orbit_inclusion}.

    The fact that $\cO$ is a smooth embedding is proved in \Cref{app:mfd_smooth_maps} (\Cref{app:topologies}). According to \cite[Chapter~1]{Kirillov}, the symplectic form $\omega$ on the coadjoint orbit $\mathcal O_{(\beta_y, \delta_y)}$ is determined pointwise by
    \begin{align*}
        &\omega_{(\eta_x, \delta_x)} \bigl( \ad^*_{(X, f)} (\eta_x, \delta_x), \ad^*_{(Y, g)} (\eta_x, \delta_x) \bigr) 
        \coloneqq \blangle (\eta_x, \delta_x), [ (X, f), (Y, g)] \brangle \\ 
        &= \blangle (\eta_x, \delta_x), (-[X, Y], -Xg + Yf) \brangle = - \eta_x \bigl( [X, Y]_x \bigr) - (Xg - Yf)(x)
    \end{align*}
    (see \eqref{eq:Liebracket} for the second equality). Since $\mu$ intertwines the coadjoint action of $\GM$ on $\mathcal O_{(\beta_y, \delta_y)}$ with its action $\alpha^0$ \eqref{eq:alpha_0} on $T^*M$, the vector field $\ad^*_{(X, f)}$ on $\mathcal O_{(\beta_y, \delta_y)}$ is the pushforward by $\mu$ of the vector field $d\alpha^0_{(X, f)}$ on $T^*M$, where $d\alpha^0 \colon \gM \to \X(T^*M)$ is the derived action. To show that $\mu$ is a symplectomorphism, it therefore suffices to show that the symplectic form $\omega'$ on $T^*M$ satisfies \begin{align}\label{eq:symp_form_in_proof_of_diffeo}
        \omega'_{\eta_x} \bigl( d\alpha^0_{(X, f)}(\eta_x), d\alpha^0_{(Y, g)}(\eta_x) \bigr) = - \eta_x \bigl( [X, Y]_x \bigr) - (Xg - Yf)(x).
    \end{align}
    We prove this in \Cref{app:proofs}, \Cref{cor:symp_form_in_proof_of_diffeo}.
\end{proof}

\subsection{Geometric Quantization via Kirillov's Method}
\label{sec:geometric_quantization_result}

We now construct the representations $\rho^h \colon \GM \to \mathcal U(L^2(M))$ \eqref{eq:def_of_rho} via Kirillov's method of algebraic polarization and induction. We consider $h = 1$ for now; \Cref{rem:hbar_quant} explains the case of $h \neq 1$. Some facts regarding the structure of diffeomorphism groups are needed which, to the best of our knowledge, appear only for compact $M$ in the literature, e.g.\ in \cite{Sch23} and the sources therein. 

\begin{setup}
Fix $x \in M$ and consider the groups
\begin{align*}
    D_x:=\{\phi \in \Diff_c(M) : \phi(x)=x\} \quad \text{and} \quad H_x := \{(\phi,f) \in \GM : \phi \in D_x\}.
\end{align*}
Note that $H_x \subset \GM$ is a subgroup such that $H_x = D_x \ltimes_\gamma C_c^\infty (M,\R)$ and that there is a natural homeomorphism with respect to the function space topologies from \Cref{app:topologies}
\begin{equation}\label{eq:homeo}\GM /H_x \rightarrow \DiffcM/D_x , \quad [(\phi, f)] \mapsto [\phi],\end{equation}
where the brackets denote equivalence classes.
As $\Diff_c(M) \subset C^\infty (M,M)$ is an open subset, \cite[Corollary 5.7]{Mic80}, the smooth submersion $\ev_x$ (cf.\ \Cref{ev_smooth_subm}) restricts to a submersion $\ev_x \colon \Diff_c(M) \rightarrow M$ such that $D_x = \ev_x^{-1}(x)$. The regular value theorem \cite[Theorem E]{glo15} implies that $D_x$ is a split submanifold of finite codimension. The embedding $\iota \colon D_x \rightarrow \Diff_c(M)$ then induces an embedding $H_x = D_x \ltimes_\gamma C_c^\infty (M,\R) \rightarrow \GM$, showing that $H_x$ is also a split submanifold of finite codimension. We conclude that $D_x$ and $H_x$ are Lie subgroups of $\Diff_c(M)$ and $\GM$, respectively. 
\end{setup}

The next statement collects essential Lie theoretic facts for $D_x$, $H_x$ and $\GM$.

\begin{prop}\label{Prop:LT_of_GM}
Fix $x \in M$. Then, the following holds.
\begin{enumerate}[label=\rm (\alph*)]
    \item The Lie algebras of $D_x$ and $H_x$ are $\mathfrak{d}_x := \Lf(D_x)= \{X \in \XcM : X_x =0\}$ and $\mathfrak{h}_x := \Lf(H_x)=\{(X,f)\in \gM : X_x=0\}$.
    \item The Lie groups $D_x,H_x$ and $\GM$ are $C^k$-regular\footnote{Recall that a Lie group $G$ is called $C^k$-regular for $k \in \N_0$ if for every $C^k$-map $\xi \colon [0,1]\rightarrow \Lf(G)$ the initial value problem in the Lie group $G$,\[\begin{cases} \dot{\eta}(t) = T\lambda_{\eta(t)}(\xi(t)) \\ \eta(0) = \mathbf{1}_G\end{cases}\] admits a solution (where $\lambda_g(h)=gh$ is left translation) and the evolution map $\evol \colon C^k([0,1],\Lf(G)) \rightarrow G, \xi \mapsto \eta(1)$ is smooth.} for every $k \in \N_0$.
    \item There is a unique manifold structure on $\DiffcM/D_x \cong \GM /H_x$ (cf.\ \eqref{eq:homeo}) for which the quotient maps $p \colon \Diff_c (M) \rightarrow \Diff_c (M)/D_x$ and $q \colon \GM \rightarrow \GM /H_x$ are submersions.
    \item The evaluation $\ev_x \colon \Diff_c(M) \to M, \ev_x(\phi)= \phi(x)$ descends to a diffeomorphism $\Diff_c(M) /D_x \cong \GM/H_x \cong M$.
\end{enumerate}
\end{prop}

\noindent The reason the statement of \Cref{Prop:LT_of_GM} is rather involved and that there is indeed something non-trivial to prove is rooted in the fact that $\GM$ and $\Diff_c(M)$ are not Banach Lie groups, but modeled on more general locally convex spaces.

\begin{proof}[Proof of \Cref{Prop:LT_of_GM}]
(a) \vspace{.4ex} The statement for $H_x$ follows immediately once the formula for the Lie algebra of $D_x$ is established. Since $D_x = \ev_x^{-1}(x)$ and $\ev_x$ is a submersion by \Cref{ev_smooth_subm}, a submersion argument, \cite{glo15}, yields $\Lf(D_x)=T_{\id_M}D_x = \ker  T_{\id_M} \ev_x$. The tangent map of $\ev_x$ is computed as in the proof of \Cref{ev_smooth_subm} and is again an evaluation, which yields the identification.

\medskip
\noindent (b) \vspace{.4ex} Recall first that $\Diff_c (M)$ is $C^k$-regular for every $k\in \N_0$ (see e.g.\ \cite[p.5]{glo12}, \cite[Theorem 43.1]{KM97}) and the evolution map $\evol_{\Diff_c(M)}$ is given by the time-$1$ flow map $\Fl^X_1$ of a time dependent (compactly supported) vector field $X$. Now, as $D_x$ is a Lie subgroup, it will be $C^k$-regular if $\evol_{\Diff_c(M)} (\xi) \in D_x$ for every $C^k$-curve $\xi$ taking values in $\mathfrak{d}_x$. This is obvious by (a), as the flow of a vector field vanishing at a point fixes that point. Now, 
\begin{align}\label{eq:extension}D_x \rightarrow D_x \ltimes C^\infty_c (M,\R) \rightarrow C^\infty_c(M,\R)
\end{align}
is an extension of the $C^k$-regular Lie group $D_x$ with the abelian Lie group $C_c^\infty(M,\R)$. The latter is complete, so by \cite[Theorem C (c)]{glo12} it is also $C^k$-regular for every $k \in \N_0$. It follows that $H_x$ is $C^k$-regular, since $C^k$-regularity of Lie groups is an extension property \cite[Theorem 38.6]{KM97}. The same argument works for $\GM$ by replacing $D_x$ in \eqref{eq:extension} with $\Diff_c(M)$.

\medskip
\noindent (c) and (d) \vspace{.4ex} There is an obvious identification of the quotient sets $\Diff_c(M)/D_x$ and $\GM/H_x$. With the notation from the statement of the Proposition, we obtain
\[
\begin{tikzcd}
    \Diff_c (M) \ar[r, hookrightarrow] \ar[d,"p"]& \GM \ar[d,"q"] \ar[r,two heads] & \Diff_c(M) \ar[d,"p"] \ar[dr, shift left=1ex, "\ev_x"] \\ \Diff_c(M)/D_x \ar[r] &  \GM/H_x \ar[r] & \Diff_c(M)/D_x \ar[r, "\varphi"]& M
\end{tikzcd}
\]
where the upper row is first inclusion, then projection, and the lower row is given by the canonical identification maps and then $\varphi(\psi D_x)=\psi(x)$. We know that $D_x \subset \Diff_c(M)$ and $H_x \subset \GM$ are $C^k$-regular split Lie subgroups ($k\in \N_0$) of finite codimension. The quotients therefore have a necessarily unique manifold structure turning the canonical quotient maps $p,q$ into submersions by \cite[Theorem G (a)]{glo15}. As $M$ is connected, $\Diff_c(M)$ acts transitively on $M$ (see e.g.\ \cite{MV}) and thus $\ev_x$ is surjective. Hence,  the induced map $\varphi$ is a bijection. So all maps on the bottom row are bijective. Also $\ev_x$ is a submersion by \Cref{ev_smooth_subm}, hence smoothness of all maps and their inverses on the bottom row follows from the smoothness of the top row by a submersion argument.  \qedhere
\end{proof}

Let $x \in M$ and fix $\eta_x \in T^*_x M$, so that $\mu(\eta_x) = (\eta_x, \delta_x) \in \mu(T^*M) \subset \gM^*$. Recall that a \textit{real algebraic polarization} of $(\eta_x, \delta_x)$ is a Lie subalgebra $\mathfrak h \subset \gM$ for which $(\eta_x, \delta_x)|_{[\mathfrak h, \mathfrak h]} = 0$ and $\mathrm{codim}_{\gM} \mathfrak h = d$ (where $d=\text{dim } M$). The polarization is called \textit{admissible} if it is invariant under the stabilizer of $(\eta_x, \delta_x)$ (with respect to the coadjoint action, which is $\alpha^0$ by \Cref{thm:coadjoint_orbit}), i.e.\ if for every $(\phi, f) \in \GM$, $\alpha^0_{(\phi, f)}(\eta_x) = \eta_x$ implies that $\alpha^0_{(\phi, f)}(\mathfrak h) = \mathfrak h$.

\begin{prop}
    Let $x \in M$ and fix $\eta_x \in T^*_x M \subset \gM^*$. Then, the Lie subalgebra
    \begin{align*}
        \mathfrak h_x = \bigr\{ (X, f) \in \gM : X_x = 0 \bigl\}
    \end{align*}
    is an admissible real algebraic polarization of $\eta_x$.
\end{prop}
\begin{proof}
By \Cref{Prop:LT_of_GM} (d), $\mathrm{codim}_{\gM} \mathfrak h_x = d$. Since $\mathfrak h_x$ is a Lie algebra, $[X, Y]_x = 0$ whenever $X_x = Y_x = 0$. By \Cref{eq:Liebracket} for the Lie bracket on $\gM$, this implies that $(\eta_x, \delta_x)|_{[ \mathfrak h_x, \mathfrak h_x]} = 0$. The stabilizer of $(\eta_x, \delta_x)$ is defined by the condition that $(\phi, f) \in \GM$ satisfy $\phi(x) = x$ and $\eta_x = (\phi^{-1})^* \eta_x - df_x$. Using only the condition that $\phi(x) = x$ and \Cref{lem:adjoint_action}, we see that $\mathfrak h_x$ is invariant under the stabilizer, hence admissible.
\end{proof}

The fact that $\mathfrak h_x$ is a real algebraic polarization of $(\eta_x, \delta_x)$ is also equivalent to the statement that the map $\mathfrak h_x \to \R$ given by $(X, f) \mapsto \langle (\eta_x, \delta_x), (X, f) \rangle = f(x)$ is a Lie algebra representation of $\mathfrak h_x$, which is clear. Up to a constant, this is the derived representation of $H_x \to \C$, $(\phi, f) \mapsto e^{-2\pi i f(x)}$. Kirillov's quantization scheme is now carried out by constructing the induced representation of $\GM$.\footnote{Kirillov \cite{Kirillov} would induce from $(\phi, f) \mapsto e^{if(x)}$. We introduce the constant $-2\pi$ for consistency with \cite{Landsman_Topics}.} Kosyak \cite{Kosyak} has developed a theory of induced representations for infinite-dimensional Lie groups, which relies on the existence of a section $M \to \GM$; the particulars of our situation allow us to sidestep the section by following a procedure from Folland \cite[Section~6.1]{Folland}. 

\begin{define} \label{def:C_x}
    Fix $x \in M$. We denote by $\mathcal C_x$ the vector space of all functions $\Psi \in C(\GM, \C)$ for which $q(\supp\, \Psi) \subset M$ is compact and $\Psi((\phi,f)(\theta, g)) = e^{2\pi i g(x)} \Psi((\phi, f))$ for all $(\phi, f) \in \GM$ and $(\theta, g) \in H_x$.    
\end{define}

\noindent The following lemma shows that $\mathcal C_x$ can be identified with $C_c(M) \coloneqq C_c(M, \C)$ (\Cref{setup:R_or_C}) in such a manner that left translation on $\GM$ corresponds to the action $\rho^1$ ($h = 1$) from \eqref{eq:def_of_rho}, without the Radon--Nikodym derivative.

\begin{lemma}
    \label{lem:Kirillov_quant}
    Fix $x \in M$. For each $\Psi \in \mathcal C_x$, the prescription 
    \begin{align*}
        M \ni y \mapsto \psi(y) \coloneqq \Psi\bigl( (\phi, 0) \bigr) \quad \text{where } \phi \in \Diff_c(M) \text{ is such that } \phi(x) = y
    \end{align*}
    defines a continuous function $\psi \colon M \to \C$. Moreover, $\Psi \mapsto \psi$ is a bijection $\mathcal C_x \cong C_c(M)$, and the translate $L_{(\phi,f)} \Psi \colon (\theta, g ) \mapsto \Psi( (\phi, f)^{-1}(\theta, g))$ maps to $y \mapsto \psi(\phi^{-1}(y)) e^{-2\pi i f(y)}$.
\end{lemma}

\begin{proof}
    To see that $\psi$ is well-defined, note that if $\phi(x) = \theta(x)$, then 
    \begin{align*}
        \Psi \bigl((\phi, 0) \bigr) = \Psi \bigl( (\theta, 0)(\theta^{-1} \circ \phi, 0) \bigr) = \Psi\bigl( (\theta, 0 \bigr) \bigr)
    \end{align*}
    since $(\theta^{-1} \circ \phi, 0) \in H_x$ and $\Psi \in \mathcal C_x$. Since $\psi \circ p= \hat{\Psi}$, where $\hat{\Psi}$ is the composition $\DiffcM \hookrightarrow \GM \xrightarrow\Psi \C$, continuity of $\psi$ follows from continuity of $\hat\Psi$ and the fact that $p$ is a submersion to $M \cong \DiffcM /D_x$ (\Cref{Prop:LT_of_GM} (d)). For bijectivity of the correspondence $\Psi \mapsto \psi$, one checks that its inverse is given by 
    \begin{align*}
        \Psi \bigl( (\phi, f) \bigr) \coloneqq \psi \bigl(\phi(x) \bigr) e^{2\pi i f(\phi(x))},
    \end{align*}
    using that $\Psi ((\phi, f)) = \Psi ((\phi, 0)(\id_M, f \circ \phi) ) = \Psi ((\phi, 0)) e^{2\pi i f(\phi(x))}$ for any $\Psi \in \mathcal C_x$. The remaining verifications are straightforward calculations.
\end{proof}

\noindent Folland's construction proceeds as follows:\ (i) equip $\GM /H_x$ with a quasi-invariant measure $\mu$ (i.e.\ the pullback $L_{(\phi, f)}^*\mu$ of $\mu$ by left translation $L_{(\phi, f)}$ by any element $(\phi, f) \in \GM$ is absolutely continuous with respect to $\mu$); (ii) note that the function $\GM \to \C$ given by
\begin{align*}
    (\phi, f) \mapsto \Psi_1 \bigl( (\phi, f) \bigr) \overline{\Psi_2 \bigl( (\phi, f) \bigr)} \quad \text{for } \Psi_1, \Psi_2 \in \mathcal C_x
\end{align*}
(where the bar denotes complex conjugation) descends to a function $\Psi_1 \overline{\Psi_2}$ on $\GM/H_x$; (iii) define the inner product $\mathcal C_x \times \mathcal C_x \to \C$,
\begin{align}
    \label{eq:inner_prod}
    \langle \Psi_1, \Psi_2 \rangle \coloneqq \int_{\GM/H_x} \bigl( \Psi_1 \overline{\Psi_2} \bigr)(y)\,d\mu(y),
\end{align}
and take the Hilbert space completion $\overline{\mathcal C_x}$; (iv) define the unitary operators $\rho_{(\phi, f)}$ on $\mathcal C_x$ by
\begin{align*}
    \rho_{(\phi, f)} \Psi = \sqrt{ \frac{d L_{(\phi, f)}^* \mu}{d\mu} } L_{(\phi, f)} \Psi \quad \text{for } (\phi, f) \in \GM
\end{align*}
and extend to $\overline{\mathcal C_x}$ by continuity.

By \Cref{Prop:LT_of_GM}, $\GM /H_x \cong M$ as smooth manifolds, and the Riemannian measure $V_g$ is quasi-invariant under $\GM$ by \Cref{lem:Radon_Nikodym_formula}. The inner product \eqref{eq:inner_prod} on $\mathcal C_x \cong C_c(M)$ (\Cref{lem:Kirillov_quant}) with respect to this measure is clearly the $L^2$ inner product with respect to $V_g$, so the Hilbert space completion can be identified with $L^2(M)$. By the second part of \Cref{lem:Kirillov_quant}, the operators $\rho_{(\phi, f)}$ correspond to the operators $\rho^1_{(\phi, f)}$ from \Cref{eq:def_of_rho}. This proves: 

\begin{theorem}
    \label{thm:Kirillovs_quantization}
    The induction of $H_x \to \C$, $(\phi, f) \mapsto e^{-2\pi i f(x)}$ to $\GM$ can be identified with the representation $\rho^1 \colon \GM \to \mathcal U(L^2(M))$.
\end{theorem}

\begin{remark}[$h \neq 1$]
    \label{rem:hbar_quant}
    Let $\mathcal O \coloneqq \mu(T^*M) \subset \gM^*$. By linearity of the coadjoint action, 
    \begin{align*}
        h^{-1} \mathcal O = \bigl\{(h^{-1}\eta_x, h^{-1}\delta_x) : (\eta_x, \delta_x) \in \mathcal O \bigr\} = \bigl\{ (\eta_x, h^{-1} \delta_x) : x \in M, \eta_x \in T_x^*M \bigr\} \cong T^*M
    \end{align*}
    is also a coadjoint orbit for every $h > 0$. The subalgebra $\mathfrak h_x$ is also an admissible real algebraic polarization of $(h^{-1} \eta_x, h^{-1} \delta_x)$, and the representation of $H_x$ whose derived version is $(X, f) \mapsto -2 \pi i \langle (h^{-1}\eta_x, h^{-1}\delta_x), (X, f) \rangle = - 2\pi i f(x)/h$ is given by $(\phi, f) \mapsto e^{-2 \pi i f(x)/h}$. Thus, applying Kirillov's method (with our convention) to the coadjoint orbit $h^{-1} \mathcal O$ yields the representation $\rho^h$.
\end{remark}

\noindent The derived representations $d\rho^h$ recover geometric quantization by \Cref{prop:derived_rep}.

\section{Asymptotic Character Formulas}\label{sect:asymptotic_char}

In \Cref{sec:prelim_tangent_groupoid}, we briefly recall Lie groupoids and their convolution algebras, and we give a complete construction of the tangent groupoid $\TGM$ of a Riemannian manifold $M$. The construction is due to Connes \cite{Connes_NG}; our exposition is inspired by Landsman \cite{Landsman_Topics, Landsman_Foundations}. In \Cref{sec:asymptotic_character_formula}, we prove two structural results relating the group $\GM$ to the convolution algebra of $\TGM$ (\Cref{thm:double_centralizer} and \Cref{prop:translation_of_ops}) and we obtain the asymptotic character formulas, \Cref{prop:easy_trace_formula} and \Cref{cor:character_formula}.

\subsection{Lie Groupoids and Convolution Algebras}
\label{sec:prelim_tangent_groupoid}

In categorical language, a groupoid is just a small category where all the morphisms are isomorphisms. We give a brief definition to fix the notation and terminology.

\begin{define}[Groupoids]
    A \textit{groupoid} consists of two sets, $G$ and $G_0$ (the morphisms and the objects, respectively), along with \textit{source} and \textit{target} maps, $s \colon G \to G_0$ and $t\colon G \to G_0$, and a composition law 
    \begin{align*}
        G^{(2)} \coloneqq \{ (a, b) \in G \times G : s(a) = t(b)  \} \to G , \quad (a, b) \mapsto ab,
    \end{align*}
    satisfying the usual requirements of identities, inverses and associativity. We denote the groupoid by $G$ or by $G \arrows G_0$, where the arrows represent the source and target.
\end{define}

\noindent Given a groupoid $G \arrows G_0$, we write $G(x, y) \coloneqq \{ a \in G : s(a) = x \text{ and } t(a) = y \}$ for the set of morphisms from $x \in G_0$ to $y \in G_0$. Likewise, $G(-, y) \coloneqq t^{-1}(y)$ and $G(x,-)\coloneqq s^{-1}(x)$. The existence of identities gives us an inclusion $G_0 \to G$. As in the case of groups, we denote inversion $G \to G$ by $a \mapsto a^{-1}$.

\begin{define}[Lie Groupoids]
    A \textit{Lie groupoid} \cite{Mac05} is a groupoid $G \arrows G_0$ where $G$ and $G_0$ are smooth manifolds (possibly with boundaries, cf.\ \cite{Lee13}), the source and target maps are smooth submersions, and the composition map $G^{(2)} \to G$, the inversion $G \to G$ and the inclusion $G_0 \to G$ are all smooth.\footnote{The requirement that $s,t$ are submersions implies that $G^{(2)}$ is a smooth submanifold of $G \times G$, hence it makes sense to require smoothness of the composition $G^{(2)} \to G$.}

    To every Lie groupoid $G \arrows G_0$ there is an associated Lie algebroid $\Lf(G \arrows G_0)$ which is a vector bundle $E \to G_0$ carrying some additional structure \cite{Mac05}. 
\end{define}

The convolution algebra of a Lie groupoid is a generalization of the group algebra of a locally compact group. To define it, we need an analogue of the Haar measure—a Haar system. There are multiple definitions of Haar systems in the literature. We are following Landsman, \cite[Section~III.3.3]{Landsman_Topics} and \cite[Appendix~C.17]{Landsman_Foundations}. Additional details and proofs can be found in Paterson \cite{Paterson}. Recall (see \Cref{setup:R_or_C}) that $C_c^\infty(G) \coloneqq C_c^\infty(G, \C)$.

\begin{define}[Haar Systems]
    A \textit{Haar system} on a Lie groupoid $G \arrows G_0$ is a family $(\mu^y)_{y \in G_0}$ of measures, where $\mu^y$ is a positive measure on $G(-,y)$, satisfying the following requirements.
    \begin{itemize}
        \item In any chart of $G(-, y)$, $\mu^y$ is equivalent to the Lebesgue measure.
        \item The family is \textit{left-invariant}:\ for $f \in C_c^\infty(G)$ and $a \in G(x, y)$, we have
        \begin{align*}
             \int_{G(-, x)} f(ab)\,d\mu^x(b) = \int_{G(-, y)} f(b')\,d\mu^y(b'),
        \end{align*}
        where $d\mu^x(b)$ denotes integration over the variable $b$ in $G(-, x)$ with respect to the measure $\mu^x$, and likewise for $d\mu^y(b')$.
        \item For any $f \in C_c^\infty(G)$, the map $G_0 \to \C, \; y \mapsto \int_{G(-, y)} f(a)\,d\mu^y(a)$ is smooth.
    \end{itemize}
\end{define}

\noindent Every locally compact Lie groupoid possesses a Haar system \cite{Landsman_Topics}. 

\begin{define}[Convolution Algebras]
    Let $G \arrows G_0$ be a Lie groupoid equipped with a Haar system $(\mu^y)_{y \in G_0}$. For $f, g  \in C_c^\infty(G)$, their \textit{convolution} is defined as
    \begin{align*}
      f * g \colon G \to \C, \quad   (f * g)(a) = \int_{G(\blank,  t(a))} f(b) g(b^{-1}a) \,d\mu^{t(a)}(b).
    \end{align*}
    We also define an involution $f \mapsto f^*$ by $f^*(a) = \overline{f(a^{-1})}$ (where the bar denotes complex conjugation). These operations turn $C_c^\infty(G)$ into a $*$-algebra.
\end{define}

A generalization of the left regular representation now gives us (reduced) groupoid C$^*$-algebras. For each $x \in G_0$, we define a measure $\mu_x$ on $G(x, \blank)$ by $\mu_x(A) = \mu^x(A^{-1})$ for measurable $A^{-1} \subset G(\blank, x)$. We then obtain a $*$-algebra representation $\pi_x$ of $C_c^\infty(G)$ on $L^2(G(x, \blank), \mu_x)$ by 
\begin{align*}
    \bigl(\pi_x(f) \psi \bigr)(a) = \int_{G(\blank, t(a))} f(b)\psi(b^{-1}a) \, d\mu^{t(a)}(b), 
\end{align*}
for $\psi \in L^2(G(x, \blank), \mu_x)$ and $a \in G(x, \blank)$. The norm $\|f\|_r \coloneqq \sup_{x \in G_0} \| \pi_x(f) \|$ allows us to complete $C_c^\infty(G)$ to a C$^*$-algebra $C^*_r(G)$, the \textit{reduced groupoid C$^*$-algebra} of $G$. There is also a non-reduced groupoid C$^*$-algebra associated to $G$, denoted $C^*(G)$ and constructed by taking the supremum over a larger class of representation of $C_c(G)$. For all the groupoids we are interested in, $C^*_r(G)$ and $C^*(G)$ coincide \cite{Landsman_Ramazan}, so we will simply speak of the groupoid C$^*$-algebra $C^*(G)$ in those cases. The Lie algebroid of a Lie groupoid is itself a Lie groupoid with respect to fiberwise addition. Thus, we may also consider the convolution algebra of a Lie algebroid.

The following examples constitute building blocks for the tangent groupoid.

\begin{example}[The Pair Groupoid $M \times M$]
    \label{ex:pair_groupoid}
    For a manifold $M$, the cartesian product $M \times M$ forms a Lie groupoid $M \times M \rightrightarrows M$ by defining
    \begin{align*}
        s(x_2, x_1) = x_1, \quad t(x_2, x_1) = x_2 \quad \text{and} \quad (x_3, x_2)(x_2, x_1) = (x_3, x_1)
    \end{align*}
    for $(x_3, x_2), (x_2, x_1) \in M \times M$. Equipping $t^{-1}(x) \cong M$ with a Riemannian measure $\vol$ (independently of $x$) defines a Haar system on $M \times M$. We can identify elements of $C_c^\infty(M \times M)$ with integral kernels of operators on $L^2(M)$. Then, the convolution in $C_c^\infty(M \times M)$ corresponds to composition, the involution corresponds to the adjoint, and the canonical representations $\pi_x$ on $L^2(s^{-1}(x), \mu_x) \cong L^2(M, \vol) \eqqcolon L^2(M)$ all correspond to the usual action of $\mathcal B(L^2(M))$ on $L^2(M)$. The C$^*$-algebra $C^*(M \times M)$ can be identified with the compact operators $\cptops(L^2(M))$.\footnote{This follows by noting that $C_c^\infty(M \times M)$ is dense in $L^2(M \times M)$, that $L^2(M \times M)$ corresponds to Hilbert-Schmidt operators on $L^2(M)$ (via integral kernels) and that Hilbert-Schmidt operators are dense in $\cptops(L^2(M))$.}
\end{example}

\begin{example}[The Tangent Bundle $TM$]
    \label{ex:tangent_bundle}
    The Lie algebroid of the pair groupoid from \Cref{ex:pair_groupoid} can, as a vector bundle, be identified with the tangent bundle $TM \xrightarrow\pi M$. Equip $M$ with a Riemannian metric $g$, let $q \colon U \to \Rd$ be coordinates on $M$, and $(q, v) \colon TU \to \R^{2d}$ tangent coordinates (see \Cref{setup:coordinates}).
    A Haar system on $TM \arrows M$ is obtained by equipping $T_xM$ with the measure\footnote{This measure is coordinate independent for the same reason that the Riemannian measure is coordinate independent:\ $v^j$ and $dq^j$ transform in the same manner via change of coordinates.} $\lambda_x$ with density
    \begin{align*}
        d\lambda_x(v) = \sqrt{|\det g_q(x)|} \,d^d v = \sqrt{|\det g_q(x)|} \cdot |dv^1 \wedge \cdots \wedge dv^d |.
    \end{align*}
    The convolution in $C_c^\infty(TM)$ is given by fiberwise convolution. We can convert this to pointwise multiplication on $T^*M$ by introducing a fiberwise Fourier transform, as follows. 
    
    On $T^*M$, we have the \textit{Liouville measure} $V_L$, defined independently of the metric. In cotangent coordinates $(q, p) \colon T^*U \to \R^{2d}$, the Liouville measure is determined by the density $dV_L(q, p) = d^d p \, d^d q$. We now define the measure $\lambda^*_x$ on $T_x^*M$ in terms of the density
    \begin{align*}
        d\lambda^*_x(p) = \frac1{\sqrt{|\det g_q(x)|}} d^dp = \frac1{\sqrt{|\det g_q(x)|}} |dp^1 \wedge \cdots \wedge dp^d |.
    \end{align*}
    This measure is defined so that $\int_{T^*M} a \, dV_L = \int_M \int_{T_x^*M} a \, d\lambda^*_x \, d\vol$ for all $a \in C_c^\infty(T^*M)$. 
    Let now $a \colon T^*M \to \C$ be any function that is $\lambda^*_x$-integrable on each fiber $T_x^*M$. We define its \textit{fiberwise Fourier transform} $\fiberF a \colon TM \to \C$ by
    \begin{align*}
        \fiberF a(v_x) = \int_{T_x^*M} a(\eta_x) e^{2 \pi i \eta_x(v_x)} \, d\lambda^*_x(\eta_x).
    \end{align*}
    Our domain for the fiberwise Fourier transform will be the Paley–Wiener space 
    \begin{align*}
        \PW(T^*M) \coloneqq \bigl\{ a \in C^\infty(T^*M) : \fiberF a \in C_c^\infty(TM)\bigr\}.
    \end{align*}
    Note that $a \in \PW(T^*M)$ will be a Schwartz function on each fiber $T^*_xM$, and that the projection of $\mathrm{supp}(a)$ onto $M$ will be compact. Consequently, $\PW(T^*M) \subset C_0(T^*M)$, where $C_0(T^*M)$ denotes the space of continuous complex-valued functions on $T^*M$ vanishing at infinity. Note also that $\fiberF$ maps $\PW(T^*M)$ bijectively onto $C_c^\infty(TM)$. This follows from the observation that 
    \begin{align*}
        \fiberF^{-1} b(\eta_x) = \int_{T_xM} \, b(v_x) e^{-2\pi i \eta_x(v_x)} \, d\lambda_x(v_x) \quad \text{for } b \in C_c^\infty(TM).
    \end{align*}
    Identifying the convolution algebra $C_c^\infty(TM)$ with $\PW(T^*M)$, the convolution becomes pointwise multiplication in $\PW(T^*M)$, and the involution becomes complex conjugation. Moreover, the canonical representations $\pi_x$ of $\PW(T^*M)$ on $L^2(s^{-1}(x), \mu_x) \cong L^2(\Rd)$ become restriction to $T_x^*M \cong \Rd$ followed by pointwise multiplication. Thus, $\|a\|_r = \sup_{x \in M} \sup_{v_x \in T_x^*M} |a(v_x)| = \|a\|_\infty$, and $C^*(TM)$ can be identified with $C_0(T^*M)$. 
\end{example}

\subsubsection*{The Tangent Groupoid}

The flexibility of Lie groupoids allows us to combine a Lie groupoid $G\arrows G_0$ and its Lie algebroid $\Lf(G\arrows G_0)$ into a larger Lie groupoid—the \textit{tangent groupoid} \cite[Appendix~C.16]{Landsman_Foundations}. We will explain this construction only in the case of the pair groupoid of a Riemannian manifold $M$ (\Cref{ex:pair_groupoid}), which yields the tangent groupoid $\T M$ of $M$. The underlying set is
\begin{align*}
    \TGM \coloneqq \Bigl( \{0\} \times TM \Bigr) \coprod \Bigl( (0, 1] \times M \times M \Bigr),
\end{align*}
and we declare $(0, 1] \times M \times M \subset \T M$ to be an open set and declare the inclusion $\iota \colon V_\interior \coloneqq (0, 1] \times M \times M \to \TGM$ to be a diffeomorphism onto its image, $U_\interior \coloneqq \iota(V_\interior)$. It remains to describe the topology and smooth structure on a neighborhood of $\{0\} \times TM$. 

Let $\exp_x$ denote the Riemannian exponential map at $x \in M$. The tubular neighborhood theorem, \cite[Theorem 6.24]{Lee13} implies the existence of an open set $U_{TM} \subset TM$, in every fiber a convex neighborhood of $0$, such that 
\begin{equation}
    \label{eq:exp}
    \begin{split}
    \expTGM \colon U_{TM} &\to M \times M, \quad
    \expTGM(v_x) = \bigl(x, \exp_x(-v_x) \bigr)
    \end{split}
\end{equation}
is a diffeomorphism onto its image. This plays the role of an exponential map for the pair groupoid $M \times M$. Now,
\begin{align}
    \label{eq:V_boundary}
    V_\boundary \coloneqq \bigl\{ (h, v_x) : h v_x \in U_{TM} \bigr\}  \subset [0, 1] \times TM
\end{align}
is an open neighborhood of $\{0\} \times TM$ in $[0, 1] \times TM$, and the map
\begin{equation}
    \label{eq:Phi_boundary}
    \begin{aligned}
    \beta \colon V_\boundary \to \TGM , \quad
    (h, v_x) \mapsto \begin{cases}
        (0, v_x) &\text{for } h = 0,\\
        \bigl( h, \expTGM(h v_x) \bigr) &\text{for } h \in (0, 1] .
    \end{cases}
    \end{aligned}
\end{equation}
is injective. For each $h > 0$, $\beta$ glues as much of $TM$ onto $M \times M$ as is allowed by the tubular neighborhood theorem and the choice of $U_{TM}$. In the limit $h \to 0$, this tends to all of $TM$. 

We now declare $\beta$ to be a diffeomorphism onto its image, so that the smooth structure on $U_\boundary \coloneqq \beta(V_\boundary) \subset \TGM$ is determined by the usual smooth structure on $V_\boundary \subset TM \times [0, 1]$. The smooth structure on $\TGM$ is well-defined because $\iota^{-1} \circ \beta$ is a diffeomorphism. Indeed, it is just the map
\begin{align*}
    V_\boundary \cap \{h > 0 \} \to (0, 1] \times M \times M, \quad
    (h, v_x) \to \bigl( h, x, \exp_x( - h v_x) \bigr) 
\end{align*}
with the codomain restricted to its image, which is a diffeomorphism by the definition of $V_\boundary$ and the fact that $\expTGM$ is a diffeomorphism onto its image. 

The groupoid structure $\TGM \arrows [0, 1] \times M$ is given by the usual groupoid structures on $TM$ and $M \times M$, i.e.
\begin{align*}
    s(0, v_x) = t(0, v_x) = (0, x), \quad s(h, x, y) = (h, y) \quad \text{and} \quad t(h, x, y) = (h, x),
\end{align*}
with fiberwise addition over $\{0\} \times TM$ and the pair groupoid structure (\Cref{ex:pair_groupoid}) over $\{h\} \times M \times M$.

\subsubsection*{The Convolution Algebra of the Tangent Groupoid}

We obtain a Haar system on $\TGM$ by setting
\begin{align}
    \label{eq:Haar_system}
    \mu^{(0, x)} &= \lambda_x \quad \text{and} \quad \mu^{(h, x)} = h^{-d}\vol
\end{align}
(Examples \ref{ex:pair_groupoid} and \ref{ex:tangent_bundle}). This is proved in \Cref{app:proofs}, \Cref{subsec:app_tangent_groupoid}.

\begin{remark}
    \label{rem:hbar}
    For $h > 0$, the canonical representation $\pi_{(h, x)}$ of $C_c^\infty(\TGM)$ is on $L^2(M, h^{-d}\vol)$. From the perspective of quantization, it is more natural to work only with $L^2(M, V_g) \eqqcolon L^2(M)$. This results in an additional factor of $h^{-d}$ in our integral kernels.
\end{remark}

The convolution algebra $C_c^\infty(\TGM)$ of the tangent groupoid comes with restrictions
\begin{align*}
r_0 &: C_c^\infty(\TGM) \to C_c^\infty(TM) \cong \PW(T^*M)  \\ 
\text{and} \quad r_h &: C_c^\infty(\TGM) \to C_c^\infty(M \times M),
\end{align*}
defined by restriction to the fiber over $h \in [0, 1]$. (These function spaces are all $\C$-valued; see \Cref{setup:R_or_C}.) By \cite[Theorem~6.4]{Landsman_Ramazan}, these extend by continuity to surjective $*$-algebra homomorphisms
\begin{align*}
    r_0 &: C^*(\TGM) \to C^*(TM) \cong C_0(T^*M) \\
    \text{and} \quad r_h &: C^*(\TGM) \to C^*(M \times M) \cong_h \cptops(L^2(M))
\end{align*}
that turn $C^*(\TGM)$ into a continuous field of C$^*$-algebras over $[0, 1]$ (see \cite{Landsman_Ramazan} for more on continuous fields of C$^*$-algebras, we shall only discuss the necessary properties here). By \Cref{rem:hbar}, the $h$-dependent isomorphism $C^*(M \times M) \cong \cptops(L^2(M))$ is given by (the extension of) $F_h \in C_c^\infty(M \times M)$ corresponding to the operator with integral kernel $K_h \coloneqq h^{-d} F_h$. The continuous field structure of $C^*(\TGM)$ gives
\begin{align*}
    C^*(\TGM) \subset C_0(T^*M) \times \prod\nolimits_{h \in (0, 1]} \cptops(L^2(M)) \quad \text{as C$^*$-algebras}
\end{align*}
(where the right hand side is a product of C$^*$-algebras, which carries the supremum norm over its factors). Thus, we will denote elements of $C^*(\TGM)$ by $(a, T_h)$, with $a \in C_0(T^*M)$ and $T_h \in \cptops(L^2(M))$. To say that $(a, T_h)$ belongs to $C_c^\infty(\TGM)$ becomes the statement that
\begin{align*}
    (0, v_x) \mapsto \fiberF a(v_x) \quad \text{and} \quad (h, x, y) \mapsto h^d K_h(x, y),
\end{align*}
where $K_h$ is the integral kernel of $T_h$, defines a smooth compactly supported function on $\TGM$. In particular, $\fiberF a \in C_c^\infty(TM)$ and $K_h \in C_c^\infty(M \times M)$.

\subsection{The Asymptotic Character Formulas}
\label{sec:asymptotic_character_formula}

In this section, we study how the actions $\alpha^0$ on $T^*M$ and $\rho^h$ on $L^2(M)$ extend to the tangent groupoid, and we obtain the asymptotic character formulas. The main ingredients are the following two extension lemmas. 

\begin{lemma}
    \label{lem:scalar_extension_to_TGM}
    Let $g \colon [0, 1] \times M \to \C$ be a smooth map. Then, $\tilde g(0, v_x) = g(0, x)$ and $\tilde g(h, x, y) = g(h, x)$ defines a smooth map $\tilde g \colon \TGM \to \C$.
\end{lemma}

\begin{proof}
    It is clear that $\tilde g$ is smooth on $(0, 1] \times M \times M \subset \TGM$. Letting $\pi \colon TM \to M$ be the canonical projection and employing the diffeomorphism $\beta$ \eqref{eq:Phi_boundary}, we see that $(\tilde g \circ \beta)(h, v_x) = g(h, x) = g \circ (\id_{[0, 1]} \times\pi)(h, v_x)$, which is smooth on $V_\beta \subset [0, 1] \times TM$.   
\end{proof}

\begin{lemma}  
    \label{lem:diffeo_extension_to_TGM}
    Let $F \colon [0, 1] \times M \to M$ be a smooth map. Then,
    \begin{align*}
        \widetilde F(0, v_x) = \bigl( 0, T_{(0, x)}F(1, v_x) \bigr) \quad \text{and} \quad \widetilde F(h, x, y) = \bigl(h, F(h, x), F(0, y) \bigr) 
    \end{align*}
    with $(1, v_x) \in \R \times T_x M \cong T_{(0, x)} ([0, 1] \times M)$, defines a smooth map $\widetilde F \colon \TGM \to \TGM$.
\end{lemma}

\begin{proof}
    The map
    \begin{align*}
        G \colon V_{\beta} &\to M \times M, \quad (h, v_x) \mapsto \bigl( F(0, x), F(h, \exp_x(-h v_x)) \bigr)
    \end{align*}
    is clearly smooth and maps $\{0\} \times TM$ into the diagonal of $M \times M$. In particular, the preimage of $\expTGM(U_{TM}) \subset M \times M$ under $G$ is a neighborhood of $\{0\} \times TM$ in $V_\beta \subset [0, 1] \times TM$, which we denote by $V$. Noting that $(\expTGM)^{-1} \circ \, G(h, v_x) \in T_{F(0, x)}M$, independently of $h$, we define $\widetilde G \colon V \to TM$ via
   \begin{align*}
        \widetilde{G}(h,v_x)=\begin{cases}
        \frac1h \Bigl ((\expTGM)^{-1} \circ G (h, v_x) - (\expTGM)^{-1} \circ G(0, v_x) \Bigr) & h \neq 0  \\
        \lim_{h \to 0} \frac1h \Bigl ((\expTGM)^{-1} \circ G (h, v_x) - (\expTGM)^{-1} \circ G (0, v_x) \Bigr) & h=0.\end{cases}
    \end{align*}
    This is well-defined by smoothness of $(\expTGM)^{-1} \circ \, G \colon V \to TM$, and Schwarz's theorem (\cite[8.12.4]{Die60}; symmetry of mixed partial derivatives) implies that it is smooth. Now, $(\expTGM)^{-1} \circ \, G (0, v_x)$ is the zero vector in $T_{F(0, x)} M$, whence
    \begin{align*}
        \widetilde G(0, v_x) = -\lim_{h \to 0} \frac1h \exp_{F(0, x)}^{-1} \circ \, F \bigl(h, \exp_x(-h v_x) \bigr).
    \end{align*}
    This is the derivative of the curve $h \mapsto -\exp_{F(0, x)}^{-1} \circ \, F(h, \exp_x(-h v_x)) $ in $T_{F(0, x)}M$ at $h = 0$, which (by the chain rule) equals
    \begin{align*}
        -\bigl( T_{F(0, x)}\exp_{F(0, x)}^{-1} \circ \, T_{(0, x)}F \bigr)(1, -v_x) = T_{(0, x)}F(-1, v_x),
    \end{align*}
    where we have made the identification $\R \times T_x M \cong T_{(0, x)} ([0, 1] \times M)$. If we now conjugate $(h, v_x) \mapsto (h, \widetilde G(h, v_x))$ by $\beta$, with appropriate restrictions of domains, (using again that $(\expTGM)^{-1} \circ \, G (0, v_x) = 0 \in T_{F(0, x)}M$) we obtain the map
    \begin{align}\label{eq: conv_ext}
        (0, v_x) \mapsto \bigl(0, T_{(0, x)}F(-1, v_x) \bigr) \quad \text{and} \quad (h, x, y) \mapsto \bigl( h, F(0, x), F(h, y) \bigr)
    \end{align}
    restricted to a neighborhood of $\{0\} \times TM$ in $\TGM$. Since $\widetilde G$ is smooth, this map is smooth. Finally, conjugating this map by the smooth inversion $\TGM \to \TGM$ ($(0, v_x) \mapsto (0, -v_x)$ and $(h, x, y) \mapsto (h, y, x)$) yields $\widetilde F$.
\end{proof}

Recall the Hamiltonians $H_{(X,f)} \colon T^*M \to \R$ (for $(X, f) \in \gM$) from \Cref{prop:moment_map}. The upcoming theorem shows that the function $e^{-2\pi i H_{(X, f)}}$ on $T^*M$, defined by $e^{-2\pi i H_{(X, f)}}(\eta_x) = e^{-2\pi i H_{(X, f)}(\eta_x)}$ ($\eta_x \in T_x^*M$), is the classical limit of the operators $(\rho^h_{\exp(hX, hf)})_{h \in (0, 1]}$ on $L^2(M)$. Since $\rho^h_{\exp(hX, hf)} = \exp(-2\pi i \mathcal Q_h H_{(X, f)})$ (\Cref{prop:derived_rep}), this is a precise formulation of the expectation that $\mathcal Q_h H_{(X, f)} \to H_{(X, f)}$ as $h \to 0$. Since $H_{(X, f)} \notin C_0(T^*M)$ and $\exp(-2\pi i \mathcal Q_h H_{(X, f)}) \notin \mathcal B_0(L^2(M))$ (the former does not vanish at infinity and the latter is unitary), this cannot be formulated solely in terms of $C^*(\TGM)$. Instead, we show that multiplication by $e^{-2\pi i H_{(X, f)}}$ on $C_0(T^*M)$ and composition with $\rho^h_{\exp(hX, hf)}$ on $\mathcal B_0(L^2(M))$ (for $h \in (0, 1]$) combine to form continuous linear isomorphisms of $C^*(\TGM)$. In fact, with $C_b(T^*M)$ denoting the continuous and \textit{bounded} $\C$-valued functions on $T^*M$, we prove the stronger statement that $(e^{-2 \pi i H_{(X, f)}}, \rho^h_{\exp(hX, hf)}) \subset C_b(T^*M) \times \prod_{h \in (0, 1]} \mathcal B(L^2(M))$ belongs to a natural enlargement of the C$^*$-algebra $C^*(\TGM)$, namely its \textit{multiplier algebra}, which we define in a moment. From a C$^*$-algebraic perspective, this is perhaps the most natural structural result relating geometric quantization to the tangent groupoid, and to our knowledge it is a novel result.

Recall that a \textit{double centralizer} of a C$^*$-algebra $\mathcal A$ is a pair of bounded linear maps $L, R \colon \mathcal A \to \mathcal A$ satisfying $aL(b) = R(a)b$ for all $a, b \in \mathcal A$. The \textit{multiplier algebra} $M(\mathcal A)$ of $\mathcal A$ is the collection of all double centralizers of $\mathcal A$ and can itself be turned into a C$^*$-algebra in a natural way. In fact, $\mathcal A$ embeds into its multiplier algebra $M(\mathcal A)$, and $M(\mathcal A)$ is the largest (in terms of satisfying a universal property) unital C$^*$-algebra containing $\mathcal A$ as an \textit{essential ideal} (which means that it has a non-trivial intersection with every non-trivial ideal). It is well-known that $M(\mathcal B_0(L^2(M))) = \mathcal B(L^2(M))$ and $M(C_0(T^*M)) = C_b(T^*M)$.

\begin{theorem}
    \label{thm:double_centralizer}
    For every $(X, f) \in \gM$, we obtain a pair $(L_{(X, f)}, R_{(X, f)})$ of bounded linear maps $C^*(\TGM) \to C^*(\TGM)$ via
    \begin{align*}
        L_{(X, f)} (a, T_h) &= \bigl( ae^{-2 \pi i H_{(X, f)}}, \rho^h_{\exp(hX, hf)}T_h \bigr) \\ \text{and} \quad
        R_{(X, f)} (a, T_h) &= \bigl( a e^{-2 \pi i H_{(X, f)}}, T_h\rho^h_{\exp(hX, hf)} \bigr).
    \end{align*}
    Moreover, these preserve $C_c^\infty(\TGM)$ and constitute a double centralizer of $C^*(\TGM)$.
\end{theorem}
\begin{remark}\label{rem:left-right}
One obtains $R_{(X, f)}$ by conjugating $L_{(-X, -f)}$ with the involution on $C^*(\TGM)$. Similarly, the extensions in \Cref{lem:scalar_extension_to_TGM} and \Cref{lem:diffeo_extension_to_TGM} admit corresponding versions by conjugating with inversion. Indeed, we used the conjugated version in \eqref{eq: conv_ext} as it was more convenient for our proof. 
\end{remark}
\begin{proof}[Proof of \Cref{thm:double_centralizer}]
    In view of \Cref{rem:left-right}, it suffices to show that $L_{(X, f)} \colon C_c^\infty(\TGM) \to C_c^\infty(\TGM)$ makes sense. From there, it is clear that both mappings (as isometric linear maps) extend to $C^*(\TGM)$ and constitute a double centralizer.
    
    Let $(a, T_h) \in C_c^\infty(\TGM)$ and let $K_h$ be the integral kernel of $T_h$. Then, 
    \begin{align*}
        \fiberF \bigl( ae^{-2 \pi i H_{(X, f)}} \bigr)(v_x) &= \int_{T^*M} a(\eta_x) e^{-2 \pi i f(x)} e^{2\pi i \eta_x(v_x - X_x)} \,d\lambda^*_x(\eta_x) \\ &= e^{-2 \pi i f(x)} \fiberF a (v_x - X_x),
    \end{align*}
    while the kernel of $\rho^h_{\exp(hX, hf)} T_h$ is
    \begin{align*}
        e^{-2\pi i f_h(x) } \sqrt{\frac{d \bigl(\Fl_1^{-hX} \bigr)^* \vol }{d\vol}}(x) K_h \bigl(\Fl_1^{-hX}(x), y \bigr), \text{ where } f_h \coloneqq \int_0^1 f \circ \Fl^{-hX}_t \,dt
    \end{align*}
    (cf.\ \Cref{eq:exp_map}).
    Showing that $L_{(X, f)}(a, T_h)$ is an element of $C_c^\infty(\TGM)$ therefore amounts to showing that the function $\TGM \to \C$ defined by
    \begin{equation}
        \label{eq:proof_of_double_centralizer}
        \begin{split}
        (0, v_x) &\mapsto e^{-2\pi i f(x)} \fiberF a(v_x - X_x) \\
        \text{and} \quad (h, x, y) &\mapsto e^{-2\pi i f_h(x) } \sqrt{\frac{d \bigl( \Fl_1^{-hX} \bigr)^* \vol }{d\vol}}(x) h^d K_h \bigl(\Fl_1^{-hX}(x), y \bigr)
        \end{split}
    \end{equation}
    is smooth and compactly supported. Compact support is clear; we will show smoothness.
    
    We obtain \eqref{eq:proof_of_double_centralizer} from the smooth function $(\fiberF a, h^d K_h)$ by applying the coordinate transformation $\TGM \to \TGM$ given by
    \begin{align}
        \label{eq:proof_of_double_centralizer_2}
        (0, v_x) \mapsto (0, v_x - X_x) \quad \text{and} \quad (h, x, y) \mapsto \bigl(h, \Fl_1^{-hX}(x), y \bigr),
    \end{align}
    and then multiplying with the function $\TGM \to \C$ given by
    \begin{align}
        \label{eq:proof_of_double_centralizer_3}
        (0, v_x) \mapsto e^{-2\pi i f(x)} \quad \text{and} \quad
        (h, x, y) \mapsto e^{-2\pi i f_h(x) } \sqrt{\frac{d \bigl( \Fl_1^{-hX} \bigr)^* \vol }{d\vol}}(x).
    \end{align}
    Defining $F \colon [0, 1] \times M \to M$ by $(h, x) \mapsto \Fl_1^{-hX}(x) = \Fl_h^{-X}(x)$, \eqref{eq:proof_of_double_centralizer_2} is the smooth map $\widetilde F \colon \TGM \to \TGM$ from \Cref{lem:diffeo_extension_to_TGM} (recall that the flow of $X$ is smooth). To see that \eqref{eq:proof_of_double_centralizer_3} is smooth requires a bit more work. 

    For each $t \in [0, 1]$, define $F^t \colon [0, 1] \times M \to M$ by $(h, x) \mapsto \Fl^{-tX}_h(x)$. Then, the smooth maps $\widetilde F^t \colon \TGM \to \TGM$ from \Cref{lem:diffeo_extension_to_TGM} satisfy $\widetilde F^t(0, v_x) = (0, v_x - tX_x)$. Applying \Cref{lem:scalar_extension_to_TGM} to $g(h, x) = f(x)$, we see that the map $G\colon [0,1] \times \TGM \to \C$,
    \begin{align*}
        G \bigl(t,(0, v_x) \bigr) \mapsto f(x) \quad \text{and} \quad G \bigl(t,(h, x, y) \bigr) \mapsto  f \circ \Fl_h^{-tX}(x) = f \circ \Fl_t^{-hX}(x)
    \end{align*}
    is continuous and for each fixed $t \in [0, 1]$ smooth, since $G(t, \cdot) = \tilde g \circ \widetilde F^t (\cdot)$. By smooth parameter dependence of the integral \cite[8.11.12]{Die60}, $(h, \cdot) \mapsto \int_0^1 G (t, (h, \cdot)) \,dt$  is smooth, but this is the map $(0, v_x) \mapsto f(x)$ and $(h, x, y) \mapsto f_h(x)$. Finally, invoking \Cref{lem:scalar_extension_to_TGM} and \Cref{lem:Radon_Nikodym_formula} (along with $\Fl_1^{-hX} = \Fl_h^{-X}$) to deal with the Radon--Nikodym derivative, we conclude that \eqref{eq:proof_of_double_centralizer_3} is smooth. This proves that $L_{(X, f)}$ preserves $C_c^\infty(\TGM)$ and thus concludes the proof.
\end{proof}

We are now prepared to prove the asymptotic character formulas. Recall that $(a, T_h) \in C^\infty_c(\TGM)$ is interpreted as $a \in C_c^\infty(T^*M)$ being the dequantization, or \textit{asymptotic symbol}, of the family $(T_h)_{h \in (0, 1]}$ of operators on $L^2(M)$, all of which are trace-class because their integral kernel are in $C_c^\infty(M \times M)$. The first character formula shows that the traces of $T_h$, properly normalized, converge to the integral of $a$ over $T^*M$. This is a common and desirable result for quantization and dequantization schemes, showing that the asymptotic symbol $a$ captures the asymptotic ($h \to 0$) diagonal behavior of $T_h$, as measured by the trace. In fact, the proof illustrates the precise mechanism and provides geometric intuition:\ \eqref{eq:proof_of_trace_formula} shows that the zeroth fiberwise Fourier mode of $a$ corresponds to the diagonal of $T_h$, whence the integral over the fibers in $T^*M$ (to obtain the zeroth fiberwise Fourier mode); the integral over the diagonal of $M \times M$ (yielding the trace of $T_h$ from its integral kernel) corresponds to the integral over the base of $T^*M$. See \Cref{rem:character_formula} for further discussion of the formula.

\begin{prop}
    \label{prop:easy_trace_formula}
    For any $(a, T_h) \in C_c^\infty(\TGM)$, 
    \begin{align*}
        \lim_{h \to 0} h^d \,\tr(T_h) = \int_{T^*M} a \,dV_L,
    \end{align*}
    where $V_L$ denotes the Liouville measure on $T^*M$.
\end{prop}
\begin{proof}
    By definition of the fiberwise Fourier transform,
    \begin{align*}
        \int_{T^*M} a \,dV_L = \int_M \int_{T_x^*M} a(\eta_x) \,d\lambda^*_x(\eta_x)\,dV_g(x) = \int_M \fiberF a(0_x)\,dV_g(x).
    \end{align*}
    If $K_h$ denotes the integral kernel of $T_h$, then $\tr(T_h) = \int_M K_h(x, x)\,dV_g(x)$ (see e.g.\ \cite{Brislawn91}). The smooth submanifold $[0, 1] \times M \hookrightarrow \TGM$ given by
    \begin{align*}
        (0, x) &\mapsto (0, 0_x) \quad \text{and} \quad (h, x) \mapsto (h, x, x)
    \end{align*}
    is (topologically) closed. Indeed, under the diffeomorphism $\beta \colon V_\beta \to U_\beta \subset \TGM$, it corresponds to $[0, 1] \times 0_{TM} \subset V_\beta$, where $0_{TM}$ denotes the zero-section. 
    
    The condition $(a, T_h) \in C_c^\infty(\TGM)$ means that
    \begin{align*}
        (0, v_x) \mapsto \fiberF a(v_x) \quad \text{and} \quad (h, x, y) \mapsto h^d K_h(x, y)
    \end{align*}
    is smooth and compactly supported on $\TGM$. Consequently, its restriction to $[0, 1] \times M$,
    \begin{align}
        \label{eq:proof_of_trace_formula}
        (0, x) &\mapsto \fiberF a(0_x) \quad \text{and} \quad (h, x) \mapsto h^d K_h(x, x)
    \end{align}
    is smooth and compactly supported on $[0, 1] \times M$. The result now follows from an application of the dominated convergence theorem \cite[Theorem 22.29]{Sch97}:\ Consider the projection of the support of \eqref{eq:proof_of_trace_formula} to $\{0\} \times M$ (projecting along the unit interval), take the indicator function of the resulting set and multiply by the maximum of \eqref{eq:proof_of_trace_formula}. This compactly supported function dominates the functions $x \mapsto \fiberF a(0_x)$ and $x \mapsto h^d K_h(x, x)$ for $h \in (0 , 1]$, the latter of which converges pointwise to the former as $h \to 0$.
\end{proof}

\begin{remark}
    \label{rem:character_formula}
    We have not been able to find \Cref{prop:easy_trace_formula} in the literature, which seems surprising, given its rather elementary nature. The point of closest approach might be \cite[Remark~5.4.3]{HigsonNigel2019SatT}, where it is emphasized that the traces of $T_h$ diverge as $h \to 0$. This is of course an immediate consequence of \Cref{prop:easy_trace_formula}. Presumably, \Cref{prop:easy_trace_formula} is simply uninteresting from the point of view of index theory, which is the main arena for the tangent groupoid. In the context of quantization, however, trace formulas of this kind are important. Indeed, we have already mentioned the connection to character formulas in the context of the orbit method, and a non-asymptotic variant of \Cref{cor:character_formula} is the starting point for the field of quantum harmonic analysis \cite{Werner84}. A non-asymptotic trace formula for Weyl quantization on the tangent groupoid can be found in \cite{CariñenaEtAl}. Our result, formulated in terms of dequantization, shows that the asymptotic version holds for any quantization scheme. Similar asymptotic formulas also appear in the context of Berezin--Toeplitz quantization, see e.g.\ \cite{Berman06, Schlichenmaier00}.
\end{remark}

\begin{corollary}
    \label{cor:character_formula}
    For any $(a, T_h) \in C_c^\infty(\TGM)$ and $(X, f) \in \gM$,
    \begin{align*}
        \lim_{h \to 0} h^d \, \tr \bigl(\rho^h_{\exp(hX, hf)} T_h \bigr) = \int_{T^*M} a e^{-2\pi i H_{(X, f)}} \,dV_L.
    \end{align*}
\end{corollary}
\begin{proof}
    Combine \Cref{thm:double_centralizer} and \Cref{prop:easy_trace_formula}.
\end{proof}

\begin{remark}[Relation to Kirillov's Character Formula] \label{rem:relation_to_Kirillov}
    Consider a simply connected compact Lie group $K$. In this case, the orbit method yields a bijective correspondence between irreducible representations and regular integral coadjoint orbits. A coadjoint orbit $S$ is called \textit{regular} if it is of maximal dimension, which is equivalent to the stabilizer of any of its points (under the coadjoint action) being a maximal torus; $S$ is called \textit{integral} if $\int_C \sigma \in \Z$ for every singular 2-cycle $C$ with integral coefficients, where $\sigma$ is the canonical symplectic form on $S$.
    
    Let now $S$ be a regular integral coadjoint orbit. By linearity of the coadjoint action, $nS$ is also a coadjoint orbit for every $n \in \N$. The stabilizer of $Z \in S$ is the same as that of $nZ \in nS$, so $nS$ is also regular. The symplectic form $\omega_{nS}$ at $\eta \in nS$ is determined by $\omega_{nS}(\ad^*_X\, \eta, \ad^*_Y\, \eta) = \langle \eta, [X, Y] \rangle$, where $X, Y \in \mathfrak g$, from which it follows that the pullback of $\omega_{nS}$ to $S$ (with respect to the dilation $Z \mapsto nZ$) is $n \omega_S$, where $\omega_S$ is the symplectic form on $S$. Consequently, by pulling the integral $\int_C \sigma_{nS}$ back to $S$, we find that $nS$ is also integral. Thus, $nS$ is a regular integral coadjoint orbit for every $n \in \N$.
    
    Let $\rho^{1/n}$ denote the irreducible representation associated to $nS$, set $2d \coloneqq \dim S$ and fix $Z \in \mathfrak g$. Evaluating the character formula \eqref{eq:fin_dim_character_formula} for the orbit $nS$ at $Z/n$ and pulling the integral back to $S$, we find that $\tr(\rho^{1/n}_{\exp Z/n}) = J(Z/n)^{-1/2} n^d \int_S e^{-2\pi i \langle \phi, Z \rangle }\,dV(\phi)$, hence
    \begin{align} \label{eq:compact_char_formula}
        \lim_{n \to \infty} (1/n)^d \, \tr(\rho^{1/n}_{\exp Z/n}) = \int_S e^{-2\pi i \langle \phi, Z \rangle} \, dV(\phi)
    \end{align}
    because $J(Z/n) \to J(0) = 1$ ($J$ is analytic). Recalling the scaling argument from \Cref{rem:hbar_quant} used to obtain $\rho^h$ from Kirillov's quantization procedure, as well as the fact that $H_{(X, f)}(\eta_x) = \langle (\eta_x, \delta_x), (X, f) \rangle = \langle \mu(\eta_x), (X, f) \rangle$, we see that there is an exact analogy between the asymptotic character formula \eqref{eq:compact_char_formula} and \Cref{cor:character_formula}, where $1/n \leftrightarrow h$. The only difference is the necessary regularization supplied by $(a, T_h)$ in \Cref{cor:character_formula}.
\end{remark}

The proof of \Cref{prop:easy_trace_formula} can be adapted to yield an exact trace formula to arbitrary order in $h$. We use $\tau$-quantization, as described in \Cref{app:quantization}, to illustrate how the result applies to any particular quantization scheme compatible with the tangent groupoid; in this case, the formula take a particularly nice form.\footnote{The formula also holds for Born--Jordan quantization, which is obtained by integrating over $\tau \in [0, 1]$.} Note also that the result applies equally well to the tuples $L_{(X, f)}(a, T_h)$ and $R_{(X, f)}(a, T_h)$ obtained from \Cref{thm:double_centralizer}, whose proof contains explicit expressions for their fiberwise Fourier transforms and integral kernels in terms of those of $(a, T_h)$.

\begin{corollary} \label{cor:all_orders}
    For $(a, T_h) \in C_c^\infty(\TGM)$, let $K \colon [0, 1] \times M \to \C$ be the function defined by
    \begin{align*}
        K(0, x) = \fiberF a(0_x) \quad \text{and} \quad K(h, x) = h^d K_h(x, x)\quad  \text{for } h \in (0, 1],
    \end{align*}
    where $K_h \colon M \times M \to \C$ is the integral kernel of $T_h$. Then, $K$ is smooth and compactly supported, and, for every $N \in \N$,
    \begin{gather*}
        h^d \,\tr(T_h) = \sum_{n=0}^N c_nh^n + R_N(h) \quad \text{for all } h \in (0, 1],  \text{ where} \\
        c_n \coloneqq \frac1{n!} \int_M \partial_h^n K(0, x) \,d\vol(x) \quad \text{and} \quad R_N(h) \coloneqq \frac{1}{N!} \int_0^h \int_M \partial^{N+1}_h K(t, x) (h-t)^N \,dV_g(x)\,dt.
    \end{gather*}
    In particular, for the $\tau$-quantization ($\tau \in [0, 1]$; see \Cref{app:quantization}) $Q^\tau_ha$ of $a \in \PW(T^*M)$,
    \begin{align*}
        h^d \,\tr(Q^\tau_ha) = \int_{T^*M}a \,dV_L \quad \text{for all } h \in (0, 1].
    \end{align*}
\end{corollary}

\begin{proof}
    The function $K$ is precisely \eqref{eq:proof_of_trace_formula}, which we proved to be smooth and compactly supported. Taylor expanding $h \mapsto K(h, x)$ around $h = 0$ and differentiating under the integral sign yields the formula for $h^d \,\tr(T_h)$. For the $\tau$-quantization of $a \in \PW(T^*M)$, we use the notation from \Cref{app:quantization}. We obtain $(\beta^\tau)^{-1}(h, x, x) = (h, 0_x)$, hence the function $K(h, x) = h^d K_h^\tau a(x, x) = (\chi \cdot \fiberF a)(0_x) = \chi_1(0_x) \fiberF a(0_x) = \fiberF a(0_x)$ is independent of $h$. Thus, $c_n = 0$ for $n \geq 1$, $R_N(h) = 0$, and $c_0 = \int_M \fiberF a(0_x) \,d\vol(x) = \int_{T^*M} a \,dV_L$ (as in the proof of \Cref{prop:easy_trace_formula}).
\end{proof}

We can also use \Cref{lem:diffeo_extension_to_TGM} to generalize a result on Weyl quantization due to Landsman \cite[Theorem~5]{Landsman1993}. Our result is about dequantization (i.e.\ an arbitrary family $(T_h)_{h \in (0, 1]}$ having asymptotic symbol $a$) and is thus applicable to all quantization schemes on the tangent groupoid (\Cref{app:quantization}). Recall the action $\alpha^0$ of $\GM$ on $T^*M$ \eqref{eq:alpha_0}. The induced action on $C_0(T^*M)$ is $\alpha^0_{(\phi, f)} a \coloneqq a \circ (\alpha^0_{(\phi, f)})^{-1} = a \circ \alpha^0_{(\phi^{-1}, -f\circ \phi)}$, i.e.\
\begin{align*}
    \alpha^0_{(\phi, f)}a (\eta_x) = a \bigl( \phi^*(\eta_x + df_x) \bigr) \quad \text{for } \eta_x \in T_x^*M.
\end{align*}
Note that $\phi^*(\eta_x + df_x) \in T_{\phi^{-1}(x)}^*M$. For $(\phi, f) \in \GM$, we now define the quantum counterpart, $\alpha^h_{(\phi, f)} \colon \mathcal B(L^2(M)) \to \mathcal B(L^2(M))$, by $\alpha^h_{(\phi, f)} T = \rho^h_{(\phi, f)} T (\rho^h_{(\phi, f)})^*.$

\begin{prop}
    \label{prop:translation_of_ops}
    For any $(\phi, f) \in \GM$, the map
    \begin{align*}
        \alpha_{(\phi, f)} \colon (a, T_h) \mapsto \alpha_{(\phi, f)}(a, T_h) \coloneqq (\alpha^0_{(\phi, f)}a, \alpha_{(\phi, f)}^h T_h)
    \end{align*}
    is a $*$-automorphism of $C^*(\TGM)$ that preserves $C_c^\infty(\TGM)$.
\end{prop}
\begin{proof}
    It suffices to show that $\alpha_{(\phi, f)} \colon C_c^\infty(\TGM) \to C_c^\infty(\TGM)$ makes sense. From there, it is clear that it (as an isometric $*$-automorphism) extends to $C^*(\TGM)$. 

    Let $(a, T_h) \in C_c^\infty(\TGM)$, $q \colon U \to \Rd$ be coordinates on $M$ and $x \in U$. Taking $\widetilde q \coloneqq q \circ \phi$ as coordinates on $\phi^{-1}(U)$, the associated cotangent coordinates $(\widetilde q, \widetilde p)$ are given by $\widetilde p = p \circ (\phi^{-1})^*$. This follows from
    \begin{align*}
        \widetilde p^j ( \phi^* \eta_x) = \eta_x \Bigl( T\phi \tfrac{\partial}{\partial(q \circ \phi)^j}\big|_{\phi^{-1}(x)} \Bigr) = \eta_x \Bigl( \tfrac{\partial}{\partial q^j}\big|_x \Bigr) = p^j(\eta_x).
    \end{align*}
    Thus, with respect to the $p$ and $\widetilde p$ coordinates on $T_xM$ and $T_{\phi^{-1}(x)}M$, $\phi^*$ is the identity. Consequently, by definition of $\lambda^*_x$ and \Cref{lem:Radon_Nikodym_formula},
    \begin{align*}
        \frac{d((\phi^*)^{-1})^*\lambda_x^*}{d\lambda_{\phi^{-1}(x)}^*}(\eta_{\phi^{-1}(x)}) = \Biggl|\frac{\det g_{q \circ \phi}(\phi^{-1}(x))}{ \det g_q(x)} \Biggr|^{\frac12} = \frac{d (\phi^{-1})^* V_g}{dV_g}(x),
    \end{align*}
    which gives
    \begin{align*}
        \fiberF (\alpha^0_{(\phi, f)} a)(v_x) &= e^{-2\pi i v_x(f)}\int_{T^*_xM} a \circ \phi^*(\eta_x) e^{2\pi i \eta_x(v_x)} \,d\lambda^*_x(\eta_x) \\ &= e^{-2\pi i v_x(f)} \frac{d (\phi^{-1})^* V_g}{dV_g}(x) \fiberF a(T\phi^{-1} v_x).
    \end{align*}
    By a change of variables, the integral kernel of $\alpha_{(\phi, f)}^h T_h$ is 
    \begin{align*}
        e^{-2\pi i (f(x) - f(y))/h} \sqrt{\frac{d(\phi^{-1})^*\vol}{dV_g}}(x) \Biggl(\frac{d \phi^*V_g}{d V_g} \bigl(\phi^{-1}(y)\bigr) \Biggr)^{-\frac12}  K_h\bigl( \phi^{-1}(x), \phi^{-1}(y) \bigr).
    \end{align*}
    For $y \in U$, $\widetilde q = q \circ \phi$ are coordinates around $\phi^{-1}(y)$, whence \Cref{lem:Radon_Nikodym_formula} gives
    \begin{align*}
        \frac{d\phi^*V_g}{dV_g} \bigl(\phi^{-1}(y)\bigr) = \Biggl| \frac{\det g_{\widetilde q \circ \phi^{-1}}(y)}{\det g_{\widetilde q}(\phi^{-1}(y))} \Biggr|^{\frac12} = \Biggl(\frac{d (\phi^{-1})^*V_g}{dV_g} (y)\Biggr)^{-1}.
    \end{align*}
    Consequently, in terms of functions on $\TGM$, we obtain $\alpha_{(\phi, f)}(a, T_h)$ from $(a, T_h)$ by applying the coordinate transformation $\TGM \to \TGM$ given by
    \begin{align*}
        (0, v_x) \mapsto (0, T \phi^{-1} v_x) \quad \text{and} \quad (h, x, y) \mapsto \bigl( h, \phi^{-1}(x), \phi^{-1}(y) \bigr),
    \end{align*}
    and then multiplying with the function $\TGM \to \C$ given by
    \begin{align*}
        (0, v_x) &\mapsto e^{-2\pi i v_x(f)} \frac{d (\phi^{-1})^* V_g}{dV_g}(x) \\
        \text{and} \quad (h, x, y) &\mapsto e^{-2\pi i (f(x) - f(y))/h}\sqrt{\frac{d (\phi^{-1})^* V_g}{dV_g}(x) \frac{d (\phi^{-1})^* V_g}{dV_g}(y)}.
    \end{align*}
    Smoothness of the coordinate transformation follows from \Cref{lem:diffeo_extension_to_TGM} applied to $F(h, x) = \phi^{-1}(x)$. Smoothness of the function $\TGM \to \C$ is straightforward to show using \Cref{lem:Radon_Nikodym_formula} and the diffeomorphism $\beta$, as well as the fact that on the tangent groupoid, the difference quotients $(h, x, y) \mapsto (f(x)-f(y))/h$ extend to a smooth map via $(0, v_x) \mapsto v_x(f)$ \cite[Appendix~C.16]{Landsman_Foundations}.
\end{proof}

\section{Conclusion and Outlook}
\label{sec:conclusions}

\noindent We have shown that geometric quantization on the cotangent bundle $T^*M$ can be realized as a special instance of Kirillov's orbit method for a natural Hamiltonian action of the infinite-dimensional Lie group $\GM = \Diff_c(M) \ltimes C^\infty_c(M, \R)$ on the cotangent bundle. In particular, we showed that the moment map $T^*M \to \gM^*$ for this action is a symplectomorphic embedding onto a coadjoint orbit for a variety of locally convex topologies on the dual (\Cref{thm:coadjoint_orbit}) and that orbit method induction applied to this coadjoint orbit yields the integrated version of geometric quantization (\Cref{thm:Kirillovs_quantization}). Then we established semiclassical trace formulas (\Cref{prop:easy_trace_formula} and \Cref{cor:character_formula}), interpreted as asymptotic character formulas, by using Connes' tangent groupoid and the structure of its groupoid C$^*$-algebra. \Cref{thm:double_centralizer} and \Cref{prop:translation_of_ops} further explicate the interplay between this groupoid C$^*$-algebra, geometric quantization and the group $\GM$. We now discuss several research directions suggested by these results. The arguments are sketches, not proofs.

\subsubsection*{Many-Body Quantum Mechanics}

Goldin et al.\ \cite{GoldinEtAl80, GoldinSharp25} describe a completion of $\GM$ as ``a universal kinematic group for quantum mechanics'', arguing that the kinematics of many-particle quantum systems (even infinite collections of particles in locally finite configurations) are described by irreducible representations of this group. Our results provide a rigorous approach to the single-particle case in terms of Kirillov's orbit method quantization. A natural next step is to consider a collection of $N$ particles on the manifold $M$. The natural configuration space $\Gamma$ is the space of $N$-point subsets of $M$. Consider $M^N \coloneqq M \times \cdots \times M$ ($N$ times) and let $\Delta \subset M^N$ be the space of all tuples with two or more equal components. The symmetric group $S_N$ acts on $M^N$ and $(T^*M)^N$ by permutations, $\Gamma$ is isomorphic to the space of orbits in $M^N \backslash \Delta$ and $T^* \Gamma$ is isomorphic to the space of orbits in $T^*(M^N \backslash \Delta) \subset T^*M^N \cong (T^*M)^N$. In particular, we obtain a covering map $T^*(M^N \backslash \Delta) \to T^* \Gamma$ of degree $N!$, and the canonical symplectic form on $T^*(M^N \backslash \Delta) \subset (T^*M)^N$ descends to the canonical symplectic form on $T^*\Gamma$.

The map 
\begin{align*}
    T^*(M^N \backslash \Delta) \to \gM^*, \quad(\eta_{x_1}, \ldots, \eta_{x_N}) \mapsto \sum\nolimits_{j=1}^N (\eta_{x_j}, \delta_{x_j})
\end{align*}
clearly descends to $T^*\Gamma$, and using the fact that $\Diff_c(M)$ (for $\dim M \geq 2$) acts transitively on $N$-point subsets of $M$ \cite{MV}, we find that it is the inclusion of a coadjoint orbit. Moreover, since the symplectic structure on $T^* \Gamma$ descends from the product symplectic structure on $T^*(M^N \backslash \Delta) \subset (T^*M)^N$, the proof of \Cref{thm:coadjoint_orbit} implies that the symplectic form on $T^*\Gamma$ is the Kirillov--Kostant--Souriau symplectic form on the orbit. We conjecture that this is a smooth embedding of $T^*\Gamma$ for every topology on $\gM^*$ considered in \Cref{thm:coadjoint_orbit}, and that the irreducible representations of $\GM$ describing the kinematics of $N$-particle quantum mechanics on $M$ (as described in \cite{GoldinSharp25}) are obtained from this coadjoint orbit via Kirillov's method. The natural polarization of $\sum\nolimits_{j=1}^N (\eta_{x_j}, \delta_{x_j})$ is the Lie algebra of the subgroup
\begin{align*}
    H_{x_1, \ldots, x_N} \coloneqq \{ (\phi, f) \in \GM : \phi \text{ permutes } x_1, \ldots, x_N \},
\end{align*}
whose identity component we denote by $H_0$. The induction proceeds from a representation of $H_{x_1, \ldots, x_N}$ which restricts to $(\phi, f) \mapsto e^{-2\pi i \sum_{j=1}^N f(x_j)} \id$ on $H_0$. Writing $\chi(\phi, f) \coloneqq e^{-2\pi i \sum_{j=1}^N f(x_j)}$ (which is defined on all of $H_{x_1, \ldots, x_N}$) every such representation is of the form $\chi \otimes (\tau \circ q)$, where $q \colon H_{x_1, \ldots, x_N} \to H_{x_1, \ldots, x_N}/H_0$ is the quotient map and $\tau$ is a representation of the quotient.\footnote{Indeed, if $\sigma$ is a representation of $H_{x_1, \ldots, x_N}$ restricting to $\chi \cdot \id$ on $H_0$, then $\overline{\chi} \otimes \sigma$ is trivial on $H_0$, descends to a representation $\tau$ of the quotient, and $\chi \otimes (\tau \circ q) = \chi \otimes (\overline{\chi} \otimes \sigma) \cong \sigma$.} Thus, the additional input is a representation of the quotient group $H_{x_1, \ldots, x_N} / H_0$, and different representations yield different particle statistics (e.g.\ bosons or fermions, or anyons when $d = 2$) \cite{GoldinEtAl80, GoldinSharp25}. Moreover, we suspect that the tangent groupoid of $\Gamma$ and its C$^*$-algebra plays a similar role as in the single particle case of the present paper, yielding analogues of \Cref{cor:character_formula}, \Cref{thm:double_centralizer} and \Cref{prop:translation_of_ops}. (\Cref{prop:easy_trace_formula} and \Cref{cor:all_orders} already hold for $T^*\Gamma$, as they do not involve the group.) 

Once this has been established (if true), it is natural to investigate the remaining quantum mechanical systems discussed in \cite{GoldinSharp25}; ideally, the entire program of viewing $\GM$ as a universal kinematic group can be formulated in terms of the orbit method. From a mathematical point of view, it would then be interesting to characterize the orbits that occur in this manner and the properties that facilitate the use of tangent groupoids.

\subsubsection*{Subgroups and Homogeneous Spaces}

Our consideration of the extremely general symmetry group $\GM$ is necessitated by the general setting of Riemannian manifolds. Variants of our results will hold for geometrically motivated subgroups of $\GM$. Namely, if the underlying manifold $M$ admits extra structure, e.g.\ if $M$ is a homogeneous or symmetric space, then subgroups of $\Diff_c(M)$ preserving aspects of these structures induce subgroups of $\GM$ for which modifications of our methods yield similar results. It is also possible to trade $C_c^\infty(M, \R)$ for subspaces that are invariant under the relevant subgroups of $\Diff_c(M)$, and to relax the compact support condition. Indeed, Beltiţă and Beltiţă \cite{MR2544305} study this setup in the case where $M$ is a Lie group, $\Diff_c(M)$ is replaced by the subgroup of left translations (isomorphic to $M$ itself) and $V \subset C^\infty(M, \R)$ is a subspace invariant under left translations. We can also think of the $(2d+1)$-dimensional Heisenberg group in this manner, with $M = \R^d$:\ it is isomorphic to the semidirect product of translations (a subgroup of the diffeomorphism group) and the space of affine functions. The Schrödinger representation is then the natural analogue of $\rho^h$ and all of our results have well-known analogues. The setting of \cite{MR2544305} is the generalization to a finite-dimensional Lie group $M = G$ and a general (sufficiently rich) invariant function space.

Our results suggest another natural generalization, where $M$ is a homogeneous space of a finite-dimensional Lie group $G$. The diffeomorphism-component $G \subset \Diff(M)$ is then finite-dimensional, allowing a richer theory where one can employ Mackey's theory of systems of imprimitivity \cite{MackeyBook76} (which is unavailable to us precisely because $\Diff_c(M)$ is infinite-dimensional). This would recover Isham's approach to quantum mechanics on homogeneous spaces \cite{Isham84} within a broader framework built on the orbit method. Another classical example (beyond the Heisenberg group) discussed by Isham is the Euclidean group \cite[Section~4.9.1]{Isham84}, which corresponds to $M = S^d$, $G = \mathrm{SO}(d+1, \R)$ and $V$ the space of linear functions on $\R^{d+1}$ restricted to $S^d$. The main contribution of our suggested generalization is the treatment of infinite-dimensional function spaces, which allows for completely arbitrary (Riemannian) homogeneous spaces, as well as the resulting character formulas. This theory will capture the kinematics of single-particle quantum mechanics on any homogeneous space using the orbit method, as well as its classical limit in a C$^*$-algebraic deformation framework. 

\subsubsection*{Symplectic Groupoids of Poisson Manifolds}

A \textit{symplectic groupoid} is a Lie groupoid $G \arrows G_0$ with a symplectic structure $\omega$ for which the graph of the composition map $G^{(2)} \to G$, i.e.\ $\{(a, b, c) \in G \times G \times G : c = ab\}$, is a Lagrangian submanifold of $G \times G \times G$ with the symplectic structure $(-\omega) \oplus (-\omega) \oplus \omega$. In this case, there is a unique Poisson structure on $G_0$ for which the source and target maps are Poisson and anti-Poisson maps, respectively \cite{Weinstein87, Mac05}. The question of which Poisson manifolds arise in this manner (as the object space of a symplectic groupoid) was settled by Crainic and Fernandes \cite{CrainicFernandes04}. 

Symplectic groupoids were introduced in large part as a tool to quantize Poisson manifolds \cite{Weinstein87}. The formal (purely algebraic) deformation quantization program for Poisson manifolds was settled in full by Kontsevich \cite{Kontsevich03}; symplectic groupoids are utilized to construct C$^*$-algebraic quantization procedures. Rieffel gave a particular formulation of the C$^*$-algebraic quantization program in \cite[Questions~9~and~10]{RieffelQuestions98}; Hawkins \cite{Hawkins08} takes a different approach, de-emphasizing the continuous deformation aspect. This is an ongoing research program with many partial results and different approaches \cite{AlbertoGiovanni01, Landsman06, Hawkins08, BonechiEtAl14, Lackman24}. Now, the cotangent bundle $T^*M$ can be identified with a symplectic groupoid $T^*M \arrows M$, where both the source and target maps are the canonical projection $T^*M \to M$ (cf.\ \Cref{ex:tangent_bundle}). This is a symplectic groupoid with respect to the canonical symplectic structure on $T^*M$ and the trivial (i.e.\ vanishing) Poisson structure on $M$. Thus, cotangent bundles can be understood as the simplest examples of symplectic groupoids. The question naturally arises:\ Does Kirillov's algebraic quantization procedure, as presented in this paper, generalize to a wider class of symplectic groupoids?  

The action $\alpha_0 \colon \GM \to \Diff(T^*M)$ \eqref{eq:alpha_0} takes its image in the group of Hamiltonian symplectomorphisms of $T^*M$ by \Cref{prop:moment_map}. These symplectomorphisms moreover preserve the vertical polarization of $T^*M$ (in the sense of mapping leaves to leaves), which corresponds to our choice of real algebraic polarization of $(\eta_x, \delta_x) \in \gM^*$ (see e.g.\ \cite[Section~1.5]{Kirillov} for the relation between geometric and algebraic polarizations in the finite-dimensional case). Given a symplectic groupoid $G \arrows G_0$ and a polarization in the sense of Hawkins \cite[Definition~4.7]{Hawkins08}, a natural generalization is to consider subgroups of symplectomorphic or Hamiltonian diffeomorphisms of $G$ preserving the polarization. Other approaches to quantization of Poisson manifolds via symplectic groupoids could of course lead to different understandings of the relevant structure of $\GM$. We also note that $\Diff_c(M)$ coincides with the compactly supported Poisson maps for the trivial Poisson structure on $M$; infinite-dimensional Lie groups of Poisson diffeomorphisms have recently been studied by Smilde \cite{Smilde23}. The symplectic groupoids integrating Poisson--Lie groups (Lie groups with a compatible Poisson structure) discovered by Lu and Weinstein \cite{LuWeinstein89, LuPhD} could serve as a natural starting point for further investigation.

\appendix

\section{Manifolds of Mappings and Dual Spaces}\label{app:mfd_smooth_maps}

In this section, we recall essentials on spaces of smooth maps from a not necessarily compact manifold $M$ with values in another manifold. Our aim is then to realize the cotangent bundle $T^*M$ as an embedded submanifold of the dual space of an infinite-dimensional Lie algebra constructed from $M$.

\begin{setup}
Let $M$ and $N$ be $\sigma$-compact smooth manifolds without boundary. We write $C^\infty (M,N)$ for the set of all smooth functions from $M$ to $N$. There are two topologies (see e.g.\ \cite{HaS17}) on this set:
\begin{enumerate}
\item The \textit{compact open $C^\infty$-topology} is the initial topology with respect to \[C^\infty (M,N) \rightarrow \prod\nolimits_{k\in \N_0} C_{\mathrm{co}}(T^kM,T^kN), \quad f \mapsto (T^k f)_{k\in \N_0},\]
where the subscript denotes the usual compact open topology.
We write $C^\infty_{\mathrm{co}} (M,N)$ for this topology. It controls a function and its  derivatives on compact sets.
\item The \textit{fine very strong topology}, also called the $\mathcal{FD}$-topology in \cite{Mic80}. It controls functions and derivatives on locally finite families of compact sets.
\end{enumerate}
For non-compact $M$, the first topology does not turn $C^\infty(M,N)$ into a manifold, while the second does.  If $M$ is compact, both topologies coincide, see e.g.\ \cite{HaS17}.
\end{setup}
\begin{setup}\label{setup:indlim}
Let $M$ be a $\sigma$-compact manifold and $E \rightarrow M$ a smooth vector bundle.
We write $\Gamma(E)$ for the \emph{set of smooth sections} of $E$, and let $0 \in \Gamma(E)$ be the \emph{zero-section}.
For a compact subset $K \subset M$, consider the subspace of smooth \emph{sections supported in} $K$, i.e.
\begin{equation*}
    \Gamma_{K}(E) \coloneq \{ \sigma \in \Gamma(E) : \sigma|_{X \setminus K} = 0 \}.
\end{equation*}
Topologize $\Gamma_K(E)$ with the subspace topology $\Gamma_K(E) \subset C^{\infty}_{\mathrm{co}}(M,E)$.
Note that $\Gamma_K(E) \subset \Gamma_L(E)$ if the compact sets satisfy $K\subset L$. This inclusion is a topological embedding and the spaces $\Gamma_K(E)$ are Hausdorff locally convex spaces, \cite[Remark F.8]{glo04}.
The \emph{space of compactly supported sections} is defined as
$$\Gamma_c (E) := \bigcup_{K \subset M \text{ compact}} \Gamma_K (E).$$
Letting $K$ range over the compact subsets of $M$, using $\Gamma_K(E) \subset \Gamma_c(E)$, we topologize 
\begin{align}
    \label{defn:indlim}
    \Gamma_c(E) \coloneq \varinjlim \Gamma_K(E) 
\end{align}
as a locally convex inductive limit. The resulting topology on $\Gamma_c(E)$ is a Hausdorff locally convex vector topology by \cite[Lemma F.19 (c)]{glo04}.
Note that $\X (M) = \Gamma (TM)$. Similarly, we denote by $\X_c(M)$ and $\X_K(M)$ the spaces of vector fields with compact support and support in $K$, respectively.
\end{setup}

For the trivial bundle $M\times \R^n$ ($n \in \N$) over $M$, we have a canonical identification $\Gamma (M\times \R^n) \cong C^\infty (M, \R^n)$. This introduces a topology on the set of smooth, compactly supported functions by declaring $\Gamma_c (M\times \R^n) \cong C^\infty_c (M,\R^n)$ a homeomorphism. Similarly, but restricting to $\Gamma_K (M\times \R^n)$, we obtain the compact open $C^\infty$-topology on the image $C^\infty_K (M,\R^n)$.

Let $\pi \colon E\rightarrow M$ be a finite rank vector bundle. As \cite[Proposition 4.8 and Remark 8.6]{Mic80} show, the subspace topology on $\Gamma_c(E) \subset C^\infty (M,E)$ induced by the fine very strong topology coincides with the inductive limit topology \eqref{defn:indlim}. 
\begin{setup}\label{verystrong_convention}
The resulting topology on $C^\infty_c(M,\R^n)$ and $\Gamma_c (E)$ is called the fine very strong topology. We will always endow these spaces with this topology.
\end{setup}

For $\sigma$-compact manifolds, the fine very strong topology turns $C^\infty (M,N)$ into an infinite-dimensional smooth manifold (see e.g.\ \cite{Mic80,GaS22}). The model spaces of this manifold are locally convex spaces of sections $\Gamma_c (f^\ast TN)$ into pullback bundles, which are not Banach spaces. Hence, smoothness has to be interpreted in the sense of Bastiani calculus, cf.\ \cite{Sch23}. This means that a map is smooth if and only if locally in charts all of its (iterated) directional derivatives exist and combine to form continuous mappings. As a consequence:

\begin{lemma}\label{ev_smooth_subm}
Let $M$ and $N$ be finite-dimensional manifolds. For $x \in M$, 
\[\ev_x \colon C^\infty (M,N) \rightarrow N, \quad\ev_x (f)=f(x) \text{ is a smooth submersion.}\]
\end{lemma}

\begin{proof}
Smoothness of $\ev_x$ is established in \cite[Corollary 11.7]{Mic80}. The canonical identification $TC^\infty (M,N) \cong C^\infty (M,TN)$ from \cite{Mic80} yields an identification between $T\ev_x$ and $\widetilde{\ev}_x$, with $\widetilde{\ev}_x \colon C^\infty (M,TN)\rightarrow TN, F \mapsto F(x)$ the evaluation. Since every $T_f\ev_x$ is surjective, $\ev_x$ is a submersion by \cite[Theorem A]{glo15}. 
\end{proof}

\begin{setup}
The diffeomorphism group $$\Diff (M) := \{\phi \in C^\infty (M,M) \colon \phi \text{ is bijective and } \phi^{-1} \in C^\infty (M,M)\}$$ of a $\sigma$-compact manifold $M$ is open in the fine very strong topology \cite[Corollary 5.7]{Mic80}, and
$$\DiffcM=\{\phi \in \Diff(M) \colon \exists K \subset M \text{ compact, s.t. } \phi(x)=x,\ \forall x \in M\setminus K\}$$is an open subgroup.
This manifold structure turns both groups into Lie groups \cite[Theorem 11.11]{Mic80}.\footnote{The term Lie group here is used for a group which is also a (possibly infinite-dimensional) manifold such that the group operations become smooth maps, see \cite{Neeb06,Sch23}.} The associated Lie algebra is $\mathbf{L}(\Diff(M)) = \X_c(M)$ with the negative of the usual bracket of vector fields.
\end{setup}

By \cite[Theorem 11.4]{Mic80}, the action 
\begin{equation}
  \gamma \colon   \DiffcM \times \CcMR \rightarrow \CcMR, \quad
    \gamma(\phi, f)= f \circ \phi^{-1}
\label{eq:Lgp_act}
\end{equation}
is a Lie group action. The derived action of the Lie algebra associated to $\gamma$ is computed via the formulae \cite[Corollary 11.6 and Proposition 11.13]{Mic80}, yielding
\begin{align}\label{GM:derived_action}
\mathrm{d}\gamma \colon \XcM \times C_c^\infty (M,\R) \rightarrow C_c^\infty (M,\R), \quad\mathrm{d}\gamma (X,f) =-df\circ X.
\end{align}
Hence, as in the finite-dimensional setting, we consider the semidirect product:

\begin{setup}
For a $\sigma$-compact finite-dimensional manifold $M$, we consider the semidirect product of Lie groups (cf.\ \ref{verystrong_convention} for the topology)
$$\GM := \DiffcM \ltimes_\gamma \CcMR$$
with respect to the Lie group action \eqref{eq:Lgp_act}. As in the finite-dimensional case, $\GM$ is a Lie group and its Lie algebra is the semidirect sum
$$\gM := \mathbf{L}(\GM) = \XcM \ltimes_{\mathrm{d}\gamma} \CcMR.$$
The Lie bracket on $\gM$ is
\begin{align}\label{gm:Liebracket}
        \bigl[ (X,f), (Y, g) \bigr] =  \bigl( -[X, Y], - dg\circ X + df \circ Y \bigr)=\bigl( -[X, Y], - Xg + Yf \bigr),
\end{align}
where the bracket in the left component is the commutator of vector fields. The formula \eqref{gm:Liebracket} arises from the definition of a semidirect sum Lie algebra by plugging in the derived action \eqref{GM:derived_action} and observing that the Lie bracket of $C_c^\infty (M,\R)$ vanishes, while the bracket for $\Diff_c (M)$ is the negative of the usual bracket of vector fields (see e.g.\ \cite[Example 3.25]{Sch23} for an argument that generalizes to $\Diff_c(M)$).
\end{setup}

\subsection*{An Embedding into the Dual Space}
We continue to work with the semidirect product of the last section and want to show that $T^*M$ embeds into the dual space of its Lie algebra. Recall that a semidirect product of (topological) Lie algebras is, as a locally convex space, the product of the Lie algebra factors from which it was formed. Hence, we have the following:

\begin{lemma}\label{lem:gm_summand}
There exists $m \in \N$ such that $\gM$ is isomorphic to a complemented locally convex subspace of $C^\infty_c (M,\R^{m+1})$.
\end{lemma}

\begin{proof}
Since $M$ is paracompact and $TM$ a finite rank bundle, by \cite[Lemma 5.9 and Exercise 5E]{MaS74}, we can find a vector bundle $E \rightarrow M$ such that the Whitney sum $TM \oplus E$ is isomorphic, as a vector bundle, to a trivial bundle $M \times \R^m$. Hence, \cite[4.7.9]{Mic80} implies that $C^\infty_c (M, \R^m) \cong \Gamma_c (TM \oplus E) \cong \Gamma_c (TM) \times \Gamma_c(E) = \X_c(M) \times \Gamma_c(E)$. We realize $\gM$ as a complemented subspace via the embedding $\gM = \X_c(M) \times C_c^\infty (M,\R) \rightarrow C_c^\infty(M,\R^m)\times C_c^\infty(M,\R) \cong C_c^\infty(M,\R^m\times \R)$.
\end{proof}

As a consequence of \Cref{lem:gm_summand}, the (continuous) dual of $\gM$ can be identified with a summand of the dual of $C_c^\infty (M,\R^{m+1})$:
$$\gM^\ast = \X_c(M)^\ast \times \CcMR^\ast \subset C_c^\infty (M,\R^{m+1})^\ast.$$
There is a whole range of natural topologies on the dual $V^*$ of a locally convex space $V$, all of which lie between two extreme topologies (see e.g.\ \cite{Tre06}): 
\begin{itemize} 
\item  the \textit{weak$^*$-topology}, defined by the seminorms $V^* \ni f \mapsto |f(v)|$ for $v \in V$. 
\item  the \textit{strong topology}, defined by the seminorms 
$V^* \ni f \mapsto \sup\nolimits_{v \in B} |f(v)|
$, where $B$ ranges over the bounded subsets of $V$. 
\end{itemize}  
The strong topology is clearly stronger than the weak$^*$-topology.
Our results will hold for all locally convex topologies between (i.e.\ finer than) the weak$^*$-topology and (coarser than) the strong topology on $\gM^*$. Because $\gM$ is a Montel space, these are precisely the locally convex topologies for which $(\gM^*)^* \cong \gM$ via evaluation, cf.\ \cite[Proposition 36.9]{Tre06}. 

Endow $\R^{m+1}$ with the maximum norm $\lVert \cdot \rVert$ and write $\lVert \cdot \rVert_{\mathrm{op}}$ for the (multilinear) operator norms derived from this choice. Recall from \cite[Proposition 1.13 and Appendix C]{HaS17} that the fine very strong topology on $C_c^\infty (M,\R^{m+1})$ is constructed from the following family of seminorms, indexed by $r\in \N_0$ and the compact subsets $K \subset M$ such that $K \subset U$ for some chart $q \colon U \to \R^d$ on $M$:
\begin{align}\label{seminorm_defn}
 \lVert f \rVert_{(r,K,q)}\coloneq \sup\nolimits_{x\in K} \bigl\{\lVert f(x)\rVert,\lVert d^k(f \circ q^{-1}) (q(x) ; \cdot )\rVert_{\mathrm{op}} : 1\leq k \leq r\bigr\},
\end{align}
where $d^k(f\circ q^{-1})$ denotes the $k$th iterated derivative of a vector valued function. A local base at the origin for the fine very strong topology is obtained by taking intersections $\bigcap_{j \in J} \{ f \in C_c^\infty(M, \R^{m+1}) : \| f \|_{(r_j, K_j, q_j)} < \epsilon_j\}$, where the only restriction is that $(K_j)_{j \in J}$ is locally finite and $\epsilon_j > 0$. In particular, the seminorms \eqref{seminorm_defn} are continuous on $C_c^\infty(M, \R^{m+1})$.

We need the following version of the Arzela-Ascoli theorem. 

\begin{lemma}\label{lem:Arz_Asc}
    Let $M$ be a manifold, $q \colon U \to \Rd$ a chart on $M$, $m\in \N_0$, and suppose that $B \subset C^\infty_c(M, \R^{m+1})$ is bounded. Then, for every compact $K \subset U$ and $r \in \N_0$, we have that $\sup_{f \in B}\lVert f\rVert_{(r,K,q)} < \infty$. Furthermore, for every $N \in \N_0$ there is a constant $C(N,K)$ such that, for all $x,y \in q(K)$,
    \begin{align}\label{eq:econt_est}
        \sup\nolimits_{f \in B} \lVert d^N (f\circ q^{-1})(y; \cdot) - d^N (f\circ q^{-1})(x;\cdot) \rVert_\mathrm{op}   &\leq C(N,K) \| y - x \|,
    \end{align}
    where we interpret $\lVert \cdot \rVert_\mathrm{op}$ as $\lVert \cdot \rVert$ when $N = 0$.
\end{lemma}
\begin{proof}
    As the seminorms \eqref{seminorm_defn} are continuous on $C_c^\infty (M,\R^{m+1})$, the first statement is an immediate consequence of boundedness. For the second statement, note that the chart $q \colon U \to \Rd$ induces via restriction to $U$ and composition with $q^{-1}$ a continuous linear map 
    \begin{align*}
        C^\infty_c (M,\R^{m+1}) \xrightarrow{\text{res}} C^\infty (U, \R^{m+1}) \xrightarrow{(q^{-1})^\ast} C^\infty_{\mathrm{co}} \bigl(q(U),\R^{m+1} \bigr), \quad f \mapsto f|_U \circ q^{-1},
    \end{align*}
    where $C^\infty (U, \R^{m+1})$ carries the locally convex topology induced by the seminorms \eqref{seminorm_defn} (which is precisely the compact open $C^\infty$ topology).
    Continuous linear maps preserve bounded sets, so \cite[Example~3.9.3]{Horvath} now implies \eqref{eq:econt_est}.
\end{proof}

Recall the moment map $\mu \colon T^*M \to \gM^*$ from \Cref{prop:moment_map}.

\begin{prop}
    \label{app:topologies}
    The map $\cO \colon T^*M \to \gM^*$ is a smooth embedding for any locally convex topology between the weak$^*$ and the strong topologies on $\gM^*$.
\end{prop}

\begin{proof}
Following \Cref{lem:gm_summand} and the subsequent discussion of the continuous dual, $\gM^\ast$ is a complemented subspace of $C^\infty_c(M,\R^m \times \R)^\ast$ for some $m \in \N$. Hence, we can check that $\cO$ is a smooth immersion and a topological embedding as a mapping into $C^\infty_c(M,\R^m \times \R)^\ast \cong C^\infty_c(M,\R^m)^\ast \times C^\infty_c(M,\R)^\ast$.
Since continuity, smoothness and the immersion property are local, it suffices (for those properties) to work in a chart $q \colon U \to \R^d$ on $M$ around some generic point $x$. For ease of notation, we will then identify $U$ with $q(U)$ and $f$ with $f \circ q^{-1}$. Throughout the proof, $K \subset U$ will be a compact neighborhood of $x$. Moreover, with $E$ as in the proof of \Cref{lem:gm_summand}, we can use the embedding 
$$e \colon T U \rightarrow TM \oplus E \cong M \times \R^m$$
to identify $\eta \in T^\ast_y M\cong (\R^d)^\ast$, for $y \in U$, with $\eta \circ e^{-1} \circ \text{proj}_{e(T_y(M))} \in (\R^{m})^\ast$ (where $\text{proj}_{e(T_y(M))} \colon \{y\} \times  \R^m \rightarrow e(T_yM)$ is the projection onto the subspace). Note that this identification preserves the operator norm of $\eta$ and also convergence in $T^\ast U$. Thus, in the following, we will work with linear forms in $(\R^m)^\ast$. 

\medskip
\noindent \textbf{Step 1: $\cO$ is continuous.} Since $T^*M$ is first countable, it suffices to establish sequential continuity \cite[15.34]{Sch97}. Fix a convergent sequence $(\eta^k)_k \subset T^\ast M$ and set $x_k \coloneqq \pi_{T^\ast M}(\eta^k)$ and $\eta \coloneqq \lim \eta^k \in T_x^\ast M$. Then, $x_k = \pi_{T^\ast M}(\eta^k) \rightarrow x = \pi_{T^\ast M} (\eta)$. We establish convergence of $\mu(\eta^k) = (\eta^k, \delta_{x_k})$ to $\mu(\eta) = (\eta, \delta_x)$ in the strong topology on $\gM^\ast$, as this implies convergence in all weaker topologies. Let $B \subset \gM\subset C^\infty_c(M,\R^m) \times C^\infty (M,\R)$ be bounded. Then, for $(X, f) \in B$ and for $k$ large enough that $x_k \in K$,
\begin{align*}
    &\bigl| \blangle (\eta^k, \delta_{x_k}) - (\eta, \delta_x), (X, f) \brangle \bigr|  
    \leq \lvert \eta^k(X_{x_k}) -  \eta(X_x)\rvert + \lvert f(x_k) - f(x) \rvert\\ &\leq \lvert \eta^k(X_{x_k}) - \eta^k (X_x)\rvert + \lvert \eta^k(X_x)-\eta(X_x)\rvert + \lvert f(x_k) - f(x) \rvert \\
    &\leq \lVert \eta^k\rVert_{\mathrm{op}} \lVert X_{x_k}-X_x\rVert + \lVert \eta^k -\eta\rVert_{\mathrm{op}}\lVert X_x\rVert +  C(0,K)\lVert x_k - x\rVert \\
    &\leq  \bigl(\sup\nolimits_l \lVert \eta^l\rVert_{\mathrm{op}} + 1 \bigr) C(0, K) \|x_k - x\| + \lVert \eta^k -\eta\rVert_{\mathrm{op}}\lVert X_x\rVert 
\end{align*} 
where $C(0,K)$ is the constant from \Cref{lem:Arz_Asc}. Since $\lVert \eta^k-\eta\rVert_{\mathrm{op}}\rightarrow 0$, the last term in the final line converges to $0$ and $\sup\nolimits_l \lVert \eta^l\rVert_{\mathrm{op}} < \infty$. Since $x_k \to x$, we conclude that $\cO$ is continuous.

\medskip
\noindent \textbf{Step 2: $\cO$ is a topological embedding.} By Step 1, we know that $\cO$ is continuous, and it is clearly injective. We will show that it is also open onto its image, hence a topological embedding. It suffices to prove the claim for the weak$^*$-topology on $\gM^*$. Let $\theta \mapsto (\eta^\theta, \delta_{x_\theta})$ be a net in $\cO(T^\ast M) \subset \gM^*$ that converges to $(\eta, \delta_x) \in \cO(T^*M)$. Convergence in the weak$^\ast$-topology means that
\begin{align*}
    \blangle (\eta^\theta, \delta_{x_\theta}), (X, f) \brangle = \eta^\theta (X_{x_\theta}) + f(x_\theta)  \to \blangle (\eta, \delta_x), (X, f) \brangle = \eta(X_x) + f(x) 
\end{align*}
for all $(X, f) \in \gM$. Take $X \equiv 0$ and pick a function $f$ supported in a neighborhood $U$ of $x$ and satisfying $f(x) = 1$. Then, $f(x_\theta) \to 1$, so $x_\theta$ must eventually be contained in $U$. We can repeat this for any $f$ supported in an arbitrarily small neighborhood of $x$, hence $x_k \to x$ in $M$. Take now instead $f \equiv 0$ and $X=\partial_i$ equal to the constant $i$th coordinate vector field on an $x$-neighborhood. As $x_\theta\rightarrow x$, the $i$th component of $\eta^\theta$ converges to the $i$th component of $\eta$. Thus, $\eta^\theta \to \eta$ in $T^*M$.

\medskip
\noindent \textbf{Step 3: $\cO$ is a $C^1$-immersion.} 
We first compute the directional derivatives of $\cO$ in the weak$^\ast$-topology. We then prove that the difference quotients also converge in the strong topology. For the weak$^*$-limits, fix $(X,f) \in \gM \subset C^\infty_c (M,\R^m\times \R)$ and $\eta \in T^\ast U \subset T^\ast M$. For $\eta=(x,\eta_x) \in U \times (\R^m)^\ast$, $V=(v, S) \in \R^d \times (\R^m)^\ast$ and $t > 0$ small enough that $x+tv \in U$,
\begin{align*}
 \blangle \cO(\eta+tV)-\cO(\eta), (X,f)\brangle &=  (\eta_x +tS)(X_{x+tv})+ f(x+tv)-(\eta_x (X_x)+f(x)) \\
 &= \eta_x (X_{x+tv}-X_x) + t S (X_{x+tv}) + f(x+tv)-f(x).
\end{align*}
Dividing by $t$ and exploiting linearity, the limit becomes
\begin{align}\label{deriv_proxy}
    \blangle d\cO (\eta;V), (X,f) \brangle = \eta_x (dX(x;v))+S (X_x)+df(x;v).
\end{align}
Hence, the limit $d\cO (\eta;V) = \lim_{t \to 0} t^{-1}(\mu(\eta + tV) - \mu(\eta))$ exists in the weak$^*$-topology. To see that it also exists in the strong topology, fix a bounded set $B \subset C^\infty_c(M,\R^m\times \R)$. Then, for all $(X,f) \in B$, with the notation as above, and $t > 0$ small enough that $x+stv \in K$ for all $s \in [0, 1]$, 
\begin{align*}
&\bigl| \blangle t^{-1} (\cO(\eta+tV)-\cO(\eta)) - d\cO(\eta;V),(X,f)\brangle \bigr| \\ &\leq \lVert \eta_x\rVert_{\mathrm{op}} \lVert t^{-1}(X_{x+tv}-X_x) - dX(x;v)\rVert + |S(X_{x+tv}-X_x)|  \\ & \quad + \lVert t^{-1}(f(x+tv)-f(x))-df(x;v)\rVert
\\
&\leq \lVert \eta_x\rVert_{\mathrm{op}} \int_0^1 \lVert dX(x+stv;v)-dX(x;v) \rVert \, ds + \lVert S\rVert_{\mathrm{op}}\lVert X_{x+tv}-X_x\rVert  \\
& \quad + \int_0^1\lVert df(x+stv;v)-df(x;v)\rVert \, ds \\
&\leq t\lVert v\rVert \bigl(\lVert \eta_x\rVert_{\mathrm{op}} C(1,K)+\lVert S\rVert_{\mathrm{op}}C(0,K) + C(1, K)\bigr)\rightarrow 0 \quad \text{ as } t\rightarrow 0,
\end{align*}
uniformly in $(X, f) \in B$. Here, we used the mean value theorem in integral form (see e.g.\ \cite[Proposition 1.18]{Sch23}) for the second inequality, and $C(0, K)$ and $C(1,K)$ are as in \Cref{lem:Arz_Asc}. This shows that $\cO$ is differentiable as a map to  $\gM^\ast$ with the strong topology. As in Step 1, continuity of $d\cO$ in the strong topology follows from the estimate
\begin{align*}
    &\bigl|\blangle d\cO(\eta;V)-d\cO(\eta^k;V^k), (X,f) \brangle\bigr| 
    \\ &\leq | (\eta - \eta^k)(dX(x;v)) | + |\eta^k(dX(x; v - v^k))| + |\eta^k(dX(x;v^k) - dX(x_k;v^k))| \\ & \quad + \, |(S-S^k)(X_x)| + |S^k(X_x - X_{x_k})| \\ & \quad + \, |df(x;v-v^k)| + |df(x, v^k) - df(x_k, v^k)|
    \\&\leq \lVert \eta - \eta^k\rVert \lVert dX(x;v)\rVert + \|\eta^k \| \|dX(x,\cdot) \| \| v - v^k \| + \| \eta^k \|  \|v^k \| \| x - x_k \| C(1, K) 
    \\ & \quad + \, \|S - S^k\| \|X_x\| + \|S^k\| \|x-x_k \| C(0, K) \\ & \quad + \, \| df(x;\cdot)\|\|v - v^k\| + \|v^k\| \|x - x_k \| C(1, K)
\end{align*}
(obtained using \eqref{deriv_proxy} and omitting the subscripts on the operator norms and the $\eta$'s), which tends to $0$ uniformly in $(X, f) \in B$ as $k \to \infty$.
Now we are left to prove that $\cO$ is an immersion. Varying $(X,f)$ in \eqref{deriv_proxy}, clearly $d\cO(\eta; \cdot)$ is injective. Since $T^\ast M$ is finite-dimensional, \cite[Theorem H]{glo15} implies that $\cO$ is a $C^1$-immersion.

\medskip
\noindent \textbf{Step 4: $\cO$ is smooth.} We start again with a computation in the weak$^\ast$-topology and argue inductively for $N>1$. 
Taking iterated directional derivatives with respect to the $\eta$-variable, for $V_i = (v_i,S_i) \in \R^d \times (\R^m)^\ast$, we obtain 
\begin{align}
d^N\cO (\eta; V_1,\ldots, V_N)(X,f)\label{higher_deriv} &= \eta_x \left(d^N X(x; v_1,\ldots, v_N)\right) +d^Nf(x; v_1,\ldots,v_N) \notag \\ & \quad + \sum_{i=1}^N S_i \left(d^{N-1}X(x; v_{1}, \ldots, \widehat{v_{i}}, \ldots, v_{N})\right)
\end{align}
where, as usual, the hat indicates the omission of the argument. We conclude that the higher derivatives of $\cO$ as a map into $\gM^\ast$ with the weak$^\ast$-topology exist. 
Now arguments as in Step 1 and 3, using the higher order estimates from the Arzela-Ascoli theorem (\Cref{lem:Arz_Asc}), applied to \eqref{higher_deriv}, show that $\cO$ is a $C^N$-map into $\gM^\ast$ with the strong topology. We omit the details; the estimates follow exactly the same strategy.
\end{proof}

\section{Additional Proofs}
\label{app:proofs}
In this appendix, we prove various well-known results (used throughout the text) for which we could not find easily accessible proofs in the literature. 

\subsection*{The Derived Action $d\alpha^0$}
In this section, we calculate an explicit expression for the derived action $d\alpha^0 \colon \gM \to \Gamma(T^*M)$, where $\Gamma(T^*M)$ denotes the Lie algebra of smooth vector fields on $T^*M$. The action $\alpha^0 \colon \GM \times T^\ast M \rightarrow T^\ast M$ was described in \eqref{eq:alpha_0}. We then obtain \Cref{prop:moment_map} and \Cref{eq:symp_form_in_proof_of_diffeo} as corollaries.

Let $q \colon U \to \Rd$ be coordinates on $M$ and consider the associated cotangent coordinates $(q, p) \colon T^*U \to \R^{2d}$ (\Cref{setup:notation}). These are coordinates on $T^*M$, so we may in turn consider the associated tangent coordinates on $T(T^*M)$. We will only need these coordinates restricted to the tangent space $T_{\eta_x}(T^*M)$, where $\eta_x \in T_x^*M$ is fixed (and $x \in U$). These are given by
\begin{align}\label{eq:coordinates_tangent}
    T_{\eta_x}(T^*M) \ni \sum_{j=1}^d v^j \frac{\partial}{\partial q^j}\Big|_{\eta_x} + \sum_{j=1}^d w_j \frac{\partial}{\partial p_j}\Big|_{\eta_x} \mapsto (v^j, w_j)_{j=1}^d.
\end{align}

\begin{prop}
    \label{prop:derived_rep_app_B}
    Let $(X, f) \in \gM$ and $\eta_x \in T_x^*M$, with $X = \sum_{j=1}^d X^j \frac{\partial}{\partial q^j}$ and $\eta_x = \sum_{j=1}^d \eta_j dq^j_x$. Then, in the coordinates \eqref{eq:coordinates_tangent}, the following identity holds:
    \begin{align*}
        d\alpha^0_{(X, f)}(\eta_x) = \biggl( X^k(x), \sum_{j=1}^d - \eta_j \frac{\partial X^j}{\partial q^k}(x) - \frac{\partial f}{\partial q^k}(x)  \biggr)_{k=1}^d.
    \end{align*}
\end{prop}

\begin{proof}
Since pullbacks commute with the exterior derivative, 
\begin{align*}
    \alpha^0_{(X, f)}(\eta_x) &= (\phi^{-1})^*\eta_x - df_{\phi(x)}
    = \sum_{j=1}^d \biggl( \eta_j d(q^j \circ \phi^{-1})_{\phi(x)} - \frac{\partial f}{\partial q^j} \bigl(\phi(x)\bigr) dq^j_{\phi(x)} \biggr)
    \\ &= \sum_{j=1}^d \biggl( \sum_{k=1}^d \eta_j \frac{\partial (q^j \circ \phi^{-1})}{\partial q^k} - \frac{\partial f}{\partial q^k} \biggr) \bigl(\phi(x)\bigr) dq^k_{\phi(x)}.
\end{align*}
Using the curve $t \mapsto (\Fl^X_t, tf) \in \GM$ to compute the derived action, we find that $d\alpha^0_{(X, f)}(\eta_x)$ has coordinates
\begin{align}
    \label{eq:proof_app_B}
    \ddt \Big|_{t=0} \biggl( q^k, \sum_{j=1}^d \eta_j \frac{\partial (q^j \circ \Fl^{-X}_t)}{\partial q^k} - t \frac{\partial f}{\partial q^k}  \biggr)_{k=1}^d \bigl( \Fl^X_t(x) \bigr).
\end{align}
Clearly, $\ddt \big|_{t=0} q^k(\Fl^X_t(x)) = X_x(q^k)= X^k(x)$ and $\ddt\big|_{t=0} t \frac{\partial f}{\partial q^k} \bigl( \Fl^X_t(x) \bigr) = \frac{\partial f}{\partial q^k}(x)$. As for the remaining terms, we note that $q^j \circ \Fl^{-X}_t$ is a smooth function 
\begin{align*}
    F \colon \R \times M \to M, \quad (t, y) \mapsto \bigl(q^j \circ \Fl^{-X}_t \bigr)(y),
\end{align*}
and hence, $\ddt \big|_{t=0} \frac{\partial (q^j \circ \Fl^{-X}_t)}{\partial q_k} ( \Fl^X_t(x) )$ becomes
\begin{align*}
    \ddt \Big|_{t=0} \frac{\partial F}{\partial q^k} \bigl(t, \Fl^X_t(x) \bigr) = \frac{\partial^2 F}{\partial t \partial q^k}(0, x) + \sum_{l=1}^d X^l(x) \frac{\partial^2 F}{\partial q^l \partial q^k}(0, x).
\end{align*}
But $F(0, y) = q^j(y)$, so that $\frac{\partial^2 F}{\partial q^l \partial q^k}(0, x) = \frac{\partial (\delta^j_k)}{\partial q^l} = 0$, and
\begin{align*}
    \frac{\partial F}{\partial t}(0, y) = \frac{d}{dt}\Big|_{t=0} q^j(\Fl^{-X}_t(y)) = (-X)_y(q^j) = - X^j(y).
\end{align*}
Inserting everything back into \eqref{eq:proof_app_B} yields the desired result.
\end{proof}

\begin{corollary}
    \label{cor:app_comoment_map}
    The map
    \begin{align*}
        H \colon \gM &\to C^\infty(T^*M) \\
        (X, f) &\mapsto H_{(X, f)} \colon \eta_x \mapsto \eta_x(X_x) + f(x)
    \end{align*}
    is a comoment map for $\alpha^0$, i.e.\ $dH_{(X, f)}(\cdot) = \omega'( d\alpha^0_{(X, f)}, \cdot )$ for all $(X, f) \in \gM$, where $\omega'$ denotes the symplectic form on $T^*M$.
\end{corollary}
\begin{proof}
    In cotangent coordinates $(q, p) \colon T^*U \to \R^{2d}$, the symplectic form on $T^*M$ is given by $\sum_{k=1}^d dq^k \wedge dp_k$. Using the identity for $d\alpha^0_{(X, f)}$ from \Cref{prop:derived_rep_app_B}, 
    \begin{align*}
        \omega'_{\eta_x} \bigl( d\alpha^0_{(X, f)}(\eta_x), \cdot \bigr) = \sum_{k=1}^d X^k(x) dp_k +  \sum_{k=1}^d \biggl( \sum_{j=1}^d \eta_j \frac{\partial X^j}{\partial q^k}(x) + \frac{\partial f}{\partial q^k}(x) \biggr) dq^k.
    \end{align*}
    On the other hand, in coordinates, $H_{(X, f)}(q, p) = \sum_{j=1}^d p_j X^j(q) + f(q)$. Calculating $(dH_{(X, f)})_{\eta_x}$, we find that it equals the expression computed above.
\end{proof}

\begin{corollary} \label{cor:symp_form_in_proof_of_diffeo}
    Equation \eqref{eq:symp_form_in_proof_of_diffeo} holds. That is,
    \begin{align*}
        \omega'_{\eta_x} \bigl( d\alpha^0_{(X, f)}(\eta_x), d\alpha^0_{(Y, g)}(\eta_x) \bigr) = -\eta_x \bigl( [X, Y]_x \bigr) - (Xg - Yf)(x)
    \end{align*}
    for all $x \in M$, $\eta_x \in T_x^*M$ and $(X, f), (Y, g) \in \gM$.
\end{corollary}
\begin{proof}
    Using the expression for $\omega'_{\eta_x}( d\alpha^0_{(X, f)}(\eta_x), \cdot )$ from the proof of the previous corollary, we find that $\omega'_{\eta_x} ( d\alpha^0_{(X, f)}(\eta_x), d\alpha^0_{(Y, g)}(\eta_x) )$ equals
    \begin{align*}
        &\sum_{k=1}^d X^k(x) \biggl(-\sum_{j=1}^n \eta_j \frac{\partial Y^j}{\partial q^k}(x) - \frac{\partial g}{\partial q^k}(x) \biggr) + \sum_{k=1}^d \biggl( \sum_{j=1}^d \eta_j \frac{\partial X^j}{\partial q^k}(x) + \frac{\partial f}{\partial q^k}(x) \biggr) Y^k(x) \\
        &\quad = \sum_{j=1}^d \eta_j \bigl( -X_x(Y^j) + Y_x(X^j) \bigr) + \bigl( - X_x(g) + Y_x(f) \bigr),
    \end{align*}
    which confirms Equation \eqref{eq:symp_form_in_proof_of_diffeo}.
\end{proof}

\subsection*{Proof of \Cref{lem:Radon_Nikodym_formula}}
In this section, we prove a slight generalization of \Cref{lem:Radon_Nikodym_formula}

\begin{prop}\label{app_Radon_Nikodym}
Let $M$ and $N$ be ($\sigma$-compact) Riemannian manifolds with metrics $g_M$ and $g_N$, respectively. Let $\phi \colon M \to N$ be a diffeomorphism and $q \colon U \to \Rd$ be coordinates on $M$. Then,
\begin{align*}
        \frac{d (\phi^* V_{g_N})}{dV_{g_M}}(x) = \Biggl| \frac{\det g_N\bigl(T\phi  \frac{\partial}{\partial q^i}\big|_x, T\phi \frac{\partial}{\partial q^j}\big|_x \bigr)}{\det g_M \bigl(\frac{\partial}{\partial q^i}\big|_x, \frac{\partial}{\partial q^j}\big|_x \bigr)} \Biggr|^{\frac12}
\end{align*}
for all $x \in U$. In particular, this is a smooth function on $M$.
\end{prop}

\begin{proof}
    Let $\psi \in C_c^\infty(M)$. The change of variables formula gives
    \begin{align*}
        \int_U \psi\, d(\phi^* V_{g_N}) = \int_{\phi(U)} \psi \circ \phi^{-1} \, dV_{g_N}.
    \end{align*}
    Using $q \circ \phi^{-1}$ as coordinates on $\phi(U)$, the right hand side becomes 
    \begin{align}
        \label{eq:Radon_Nikodym_proof}
        \int_{q(U)} (\psi \circ q^{-1})(r) \biggl| \det g_N \biggl( \frac{\partial}{\partial(q \circ \phi^{-1})^i}\Big|_{\phi \circ q^{-1}(r)}, \frac{\partial}{\partial(q \circ \phi^{-1})^j}\Big|_{\phi \circ q^{-1}(r)} \biggr) \biggr|^{\frac12} \,d^dr.
    \end{align}
    Unwinding the definitions, we see that 
    \begin{align*}
        \frac{\partial}{\partial(q \circ \phi^{-1})^i}\Big|_{\phi \circ q^{-1}(r)}  = T \phi \Bigl( \frac{\partial}{\partial q^i} \Big|_{q^{-1}(r)} \Bigr).
    \end{align*}
    Denoting the expression for the Radon--Nikodym derivative by $\mathcal R \colon M \to (0, \infty)$ and expressing $\int_U \psi \mathcal R \, d V_{g_M}$ using $q$-coordinates, one obtains \eqref{eq:Radon_Nikodym_proof}. This implies that $\int_U \psi \, d(\phi^* V_{g_N}) = \int_U \psi \mathcal R \, dV_{g_M}$ for all $\psi \in C_c^\infty(M)$.
\end{proof}

\subsection*{Proof of \Cref{prop:derived_rep}}
\begin{prop}[{\Cref{prop:derived_rep}}] \label{app_derived_rep}\emph{For $(X, f) \in \gM$ and $\psi \in C_c^\infty(M)$,}
    \begin{align*}
        d\rho^h_{(X, f)}\psi  = -\frac{2\pi i }{h}  f \cdot \psi  - \biggl(X\psi + \frac12 (\divv X) \cdot \psi \biggr) = -\frac{2\pi  i}{h} \mathcal Q_h H_{(X,f)}.
    \end{align*}
\end{prop}
\begin{proof}
    Fix $x \in M$ and $\psi \in C_c^\infty(M)$ and let $q \colon U \to \Rd$ be normal coordinates at $x$, \cite[III.8]{KaN63}. Using \eqref{eq:exp_map}, the second component of the Lie group exponential differentiates as follows:
    \begin{align*}
        \ddt \Big|_{t=0} t\int_0^1 f \circ \Fl^{tX}_{-s} \, ds = \int_0^1 f \circ \id_M \,ds + 0 = f.
    \end{align*}
    Writing $g(t)$ for the matrix $g\bigl( T_x \Fl_t^{-X} \tfrac{\partial}{\partial q^i}|_x, T_x \Fl_t^{-X} \tfrac{\partial}{\partial q^j}|_x \bigr)$, \Cref{app_Radon_Nikodym} and the definition of $\rho^h$ \eqref{eq:def_of_rho} now imply that
    \begin{align*}
        \ddt \Big|_{t=0} \Bigl( \rho^h_{\exp(tX, tf)} \psi(x) \Bigr) &= -\frac{2\pi i}{h} f (x) \psi(x) - X_x \psi  + \psi(x)\ddt\Big|_{t=0} \biggl| \frac{\det g(t)}{\det g(0)} \biggr|^{\frac14} \\
        &= -\frac{2\pi i}{h} f (x) \psi(x) - X_x \psi  + \frac14\psi(x)\ddt\Big|_{t=0} \det g(t),
    \end{align*}
    where we have also used that $g(0)$ is the identity, by our choice of normal coordinates. Jacobi's formula gives
    \begin{align*}
        \ddt\Big|_{t=0} \det g(t) = \det g(0)  \cdot \mathrm{tr}\biggl( g^{-1}(0) \frac{dg}{dt}(0) \biggr) = \mathrm{tr}\biggl(\frac{dg}{dt}(0)\biggr).
    \end{align*}
    As the partial derivatives of the metric vanish at $x$, \cite[III. Proposition 8.4]{KaN63},
    \begin{align*}
        \frac{dg}{dt}(0) = g \biggl( \ddt\Big|_{t=0} T_x \Fl_t^{-X} \tfrac{\partial}{\partial q^i}|_x, \tfrac{\partial}{\partial q^j}|_x  \biggr) + g \biggl(  \tfrac{\partial}{\partial q^i}|_x,  \ddt\Big|_{t=0} T_x \Fl_t^{-X} \tfrac{\partial}{\partial q^j}|_x  \biggr).
    \end{align*}
    Using again that the metric at $x$ is the identity, the first term becomes
    \begin{align*}
        \ddt\Big|_{t=0} \bigl( T_x \Fl_t^{-X} \tfrac{\partial}{\partial q^i}|_x \bigr)(q^j) = \ddt\Big|_{t=0} \frac{\partial}{\partial q^i}\Big|_x \bigl(q^j \circ \Fl_t^{-X} \bigr) = - \frac{\partial}{\partial q^i}\Big|_x X(q^j),
    \end{align*}
    and the second term is the same with $i$ and $j$ swapped. Consequently, 
    \begin{align*}
        \ddt\Big|_{t=0} \det g(t) = \mathrm{tr}\biggl(\frac{dg}{dt}(0)\biggr) &= -2 \sum_{j=1}^d \frac{\partial}{\partial q^j}\Big|_x X(q^j)  = -2 \, (\divv X)(x), \\ \text{hence} \quad
        \ddt \Big|_{t=0} \Bigl( \rho^h_{\exp(tX, tf)} \psi(x) \Bigr) &= -\frac{2\pi i}{h} f (x) \psi(x) - X_x \psi - \frac 12\, (\divv X)(x) \psi(x).
    \end{align*}
    This proves pointwise convergence to the desired limit. Since $\rho^h_{\exp(tX, tf)} \psi$ is supported in the compact set $\supp(X) \cap \supp(\psi)$ and bounded there (independently of $t$) for sufficiently small $t$, the dominated convergence theorem \cite[Theorem 22.29]{Sch97} gives $L^2$-convergence. Strong continuity of $t \mapsto \rho^h_{\exp(tX, tf)}$ also follows straightforwardly from compactness of $X$ and the dominated convergence theorem, whence Stone's theorem \cite[Theorem~10.15]{Hall13} gives the result. 
\end{proof}

\subsection*{A Haar System on the Tangent Groupoid}
In this section, we prove that \eqref{eq:Haar_system} defines a Haar system on $\TGM$. Recalling the notation from Examples \ref{ex:pair_groupoid} and \ref{ex:tangent_bundle}, we have:

\begin{prop}\label{subsec:app_tangent_groupoid}
    A Haar system on $\TGM$ is given by setting $\mu^{(0, x)} = \lambda_x$ and $\mu^{(h, x)} = h^{-d}\vol$.
\end{prop}
\begin{proof}
    The only condition that requires some work is smoothness of the function
    \begin{align}
        \label{eq:Haar_system_1}
        (0, x) &\mapsto \int_{T_xM} f(0, v_x) \,d\lambda_x(v_x) \quad \text{and} \quad (h, x) \mapsto \int_M f(h, x, y) \,h^{-d} d\vol(y)
    \end{align}
    for $f \in C_c^\infty(\TGM)$ in the limit $h \to 0$. Using a partition of unity, we may assume that $f$ is supported in $U_{\beta} \coloneqq \beta(V_{\beta})$ and furthermore that $f \circ  \beta$ is supported in $[0, 1] \times TU$, where $q \colon U \to \Rd$ are coordinates on $M$. With respect to the associated tangent coordinates $(q, v) \colon TU \to \R^{2d}$, we have, for $h > 0$,
    \begin{align*}
        (\id_{[0, 1]}, q, v) \circ (\beta)^{-1} (h, x, y) = \Bigl(h, q(x), -\frac1h T_xq \bigl(\exp_x^{-1}(y)\bigr) \Bigr).
    \end{align*}
    Introducing the $(h, x)$-dependent coordinates
    \begin{align*}
         y \mapsto v_{(h, x)}(y) \coloneqq -\frac1h T_xq \bigl(\exp_x^{-1}(y) \bigr),
    \end{align*}
    we can rewrite the $h > 0$ part of \eqref{eq:Haar_system_1} as
   \begin{align}
        \label{eq:Haar_system_2}
        (h, x) &\mapsto \int_{\Rd} f \bigl(h, x, v_{(h, x)}^{-1}(r) \bigr) \,h^{-d} \sqrt{\bigl|\det g_{(h, x)}\bigl(v_{(h, x)}^{-1}(r) \bigr) \bigr|} \, d^dr,
    \end{align}
    where $g_{(h, x)}$ is the matrix representation of the metric in the $v_{(h, x)}$-coordinate frame. Now, $g_{(h, x)} = h^2 g_{(1, x)}$, and $v^{-1}_{(h, x)}(r) = \exp_x \circ \,(T_xq)^{-1}(-h r)$, so that
    \begin{align}
        \label{eq:Haar_system_3}
        h^{-d} \sqrt{\bigl|\det g_{(h, x)}\bigl(v_{(h, x)}^{-1}(r) \bigr) \bigr|} = \sqrt{\bigl|\det g_{(1, x)}\bigl( \exp_x \circ \,(T_xq)^{-1}(-h r) \bigr) \bigr|},
    \end{align}
    which is smooth on its domain of definition in $[0, 1] \times U \times \R^d$. Moreover, since
    \begin{align*}
        T_x(v_{(1, x)}^{-1} \circ q) = -T_{0_x} \exp_x \circ \, (T_xq)^{-1} \circ T_x q = -\id_{T_xM},
    \end{align*}
    we find that $g_{(1, x)}(x) = g_q(x)$. Thus, \eqref{eq:Haar_system_3} extends smoothly to $h = 0$, where it takes the value $\sqrt{|\det g_q(x)|}$. Finally, $f \circ \beta$ is smooth on $V_{\beta}\subset [0,1 ] \times TM$, and $(\id_{[0, 1]}, q, v) \circ\beta^{-1}(0, v_x) = ( 0, q(x), T_xq(v_x))$, so \eqref{eq:Haar_system_2} extends smoothly via
    \begin{align*}
        (0, x) &\mapsto \int_{\Rd} f \bigl(0, (T_xq)^{-1}(r)\bigr) \sqrt{|\det g_q(x)|} \, d^dr = \int_{T_xM} f(0, v_x) \,d\lambda_x(v_x). \qedhere
    \end{align*}
\end{proof}

\section{Quantization on the Tangent Groupoid}
\label{app:quantization}

In this appendix, we construct a family of quantization schemes developed by Landsman \cite{Landsman1993, Landsman_Topics, Landsman_Foundations} using the tangent groupoid. We do not make use of this material elsewhere; it is included to illustrate some quantization procedures to which our results are directly applicable. In the rest of the text, we work with the tangent groupoid as an arena for dequantization. We think of $a \in C_0(T^*M)$ as the \textit{asymptotic symbol} of a family $(T_h)_{h \in (0, 1]} \subset \mathcal B_0(L^2(M))$ if the tuple $(a, T_h) \coloneqq (a, (T_h)_{h \in (0, 1]})$ is in $C^*(\TGM)$. A \textit{quantization procedure} $\mathcal Q$ in this framework is a procedure that takes a (sufficiently nice) function $a \in C_0(T^*M)$ as input and produces a family $(\mathcal Q_ha)_{h \in (0, 1]} \subset \mathcal B_0(L^2(M))$ of operators in such a manner that $(a, \mathcal Q_ha) \in C^*(\TGM)$. Our results for arbitrary tuples $(a, T_h) \in C^*(\TGM)$ will therefore hold in particular for all tuples produced by such quantization schemes. 

The quantization procedures we will look at are indexed by $\tau \in [0, 1]$ (Landsman does the cases $\tau \in \{0, \frac12, 1\}$; this is an obvious generalization). This is a geometric generalization\footnote{To properly generalize Shubin's theory, one would have to study more general symbol classes.} of the notion of $\tau$-symbols (and the resulting $\tau$-quantization schemes) developed by Shubin \cite[Section 2.3.3]{Shubin01} in the euclidean setting. The case $\tau = \frac12$ corresponds to Weyl quantization and $\tau = 0$ corresponds to Kohn--Nirenberg quantization. On $\R^2 \cong T^*\R$, $\tau$-quantization corresponds to the ordering rule $xp \mapsto \tau \hat x \hat p + (1-\tau) \hat p \hat x$, where $(x, p)$ are coordinates on $T^*\R$ and hats denote quantization. This is thoroughly discussed in de Gosson's monograph \cite{deGossonMethods}. Born--Jordan quantization is obtained by integrating over $\tau \in [0, 1]$, as discovered by de Gosson and Luef \cite{deGossonLuef11}. 

For $\tau \in [0, 1]$, let $U_{TM}^\tau \subset TM$ be a neighborhood of the zero-section for which the map 
\begin{equation}
    \label{eq:tau_exp}
    \begin{split}
    \exp^\tau \colon U_{TM}^\tau &\to M \times M, \quad
    \exp^\tau(v_x) = \bigl(\exp_x(\tau v_x), \exp_x(-(1-\tau)v_x) \bigr)
    \end{split}
\end{equation}
is a diffeomorphism onto its image. Note that $\tau = 0$ yields the map $\exp^L$ from \eqref{eq:exp}. Similarly to \eqref{eq:V_boundary} and \eqref{eq:Phi_boundary}, we define the domain $V_\boundary^\tau \coloneqq \bigl\{ (h, v_x) : h v_x \in U_{TM}^\tau \bigr\}  \subset [0, 1] \times TM$ and the map
\begin{equation}
    \label{eq:tau_boundary}
    \begin{aligned}
    \beta^\tau \colon V_\boundary^\tau \to \TGM , \quad
    (h, v_x) \mapsto \begin{cases}
        (0, v_x) &\text{for } h = 0,\\
        \bigl( h, \exp^\tau(h v_x) \bigr) &\text{for } h \in (0, 1] .
    \end{cases}
    \end{aligned}
\end{equation}
As before, this glues an increasing amount of $TM$ onto $M \times M$ as $h \to 0$, in such a manner that the zero-section maps onto the diagonal. Geometrically, $(h, v_x)$ is mapped to $(h, y, z)$, where $y$ is the point obtained by following the geodesic through $x$ with tangent vector $v_x$ forwards for time $\tau h$ and $z$ is obtained by following it backwards for time $(1-\tau)h$. \Cref{fig:tau_quant} (a) illustrates these maps for $\tau \in \{0, \frac12, 1\}$ in the case of $M = S^1$, the circle.  

Given a sufficiently nice function $a \colon T^*M \to \C$, we take its fiberwise Fourier transform $\fiberF a \colon TM \to \C$, extend that to a function on all of $[0, 1] \times TM$ by making it independent of $h \in [0, 1]$, and then use the inverse of the diffeomorphism $\beta^\tau \colon V_\beta^\tau \to \beta^\tau(V_\beta^\tau) \subset \TGM$ to pull it back to a function on $\TGM$. For this to yield a function on all of $\TGM$, we need a smooth cut-off function $\chi \colon [0, 1] \times TM \to [0, 1]$ supported in $V_\beta^\tau$. We construct $\chi$ by choosing any smooth function $\chi_1 : TM \to [0, 1]$ supported in $U_{TM}^\tau$ that is constantly equal to one on a neighborhood of the zero-section, then defining $\chi(h, v_x) \coloneqq \chi_1(h v_x)$.

\begin{define}[$\tau$-Quantization]
    \label{def:tau_quant}
    Let $\tau \in [0, 1]$. For each $a \in \PW(T^*M)$ and $h \in (0, 1]$, we define the \textit{$\tau$-quantization} of $a$ (at parameter value $h)$ to be the operator $\mathcal Q_h^\tau a \in \cptops (L^2(M))$ with integral kernel 
    \begin{align*}
        K^\tau_h a (x, y) = \begin{cases}
        h^{-d} (\chi \cdot \fiberF a) \bigl( (\beta^\tau)^{-1}(h, x, y) \bigr) &\text{for } (h, x, y) \in \beta^\tau(V_\beta^\tau), \\
        0 &\text{otherwise}.
        \end{cases}
    \end{align*}
    That is, $\mathcal Q_h^\tau a (\psi)(x) = \int_M K_h^\tau a(x, y) \psi(y) \, d\vol(y)$ for all $\psi \in L^2(M)$ and $x \in M$.\footnote{The factor of $h^{-d}$ in $K_h^\tau a$ appears because we want to consider $\Qh^\tau a$ as an operator on $L^2(M, \vol)$, as opposed to $L^2(M, h^{-d} \vol)$ (see \Cref{rem:hbar}).}
\end{define}

\begin{figure}[!t]
  \centering
  \begin{tabular}{c@{\hspace{8ex}}c}
        \raisebox{-0.5\height}{\includegraphics[width=0.45\linewidth]{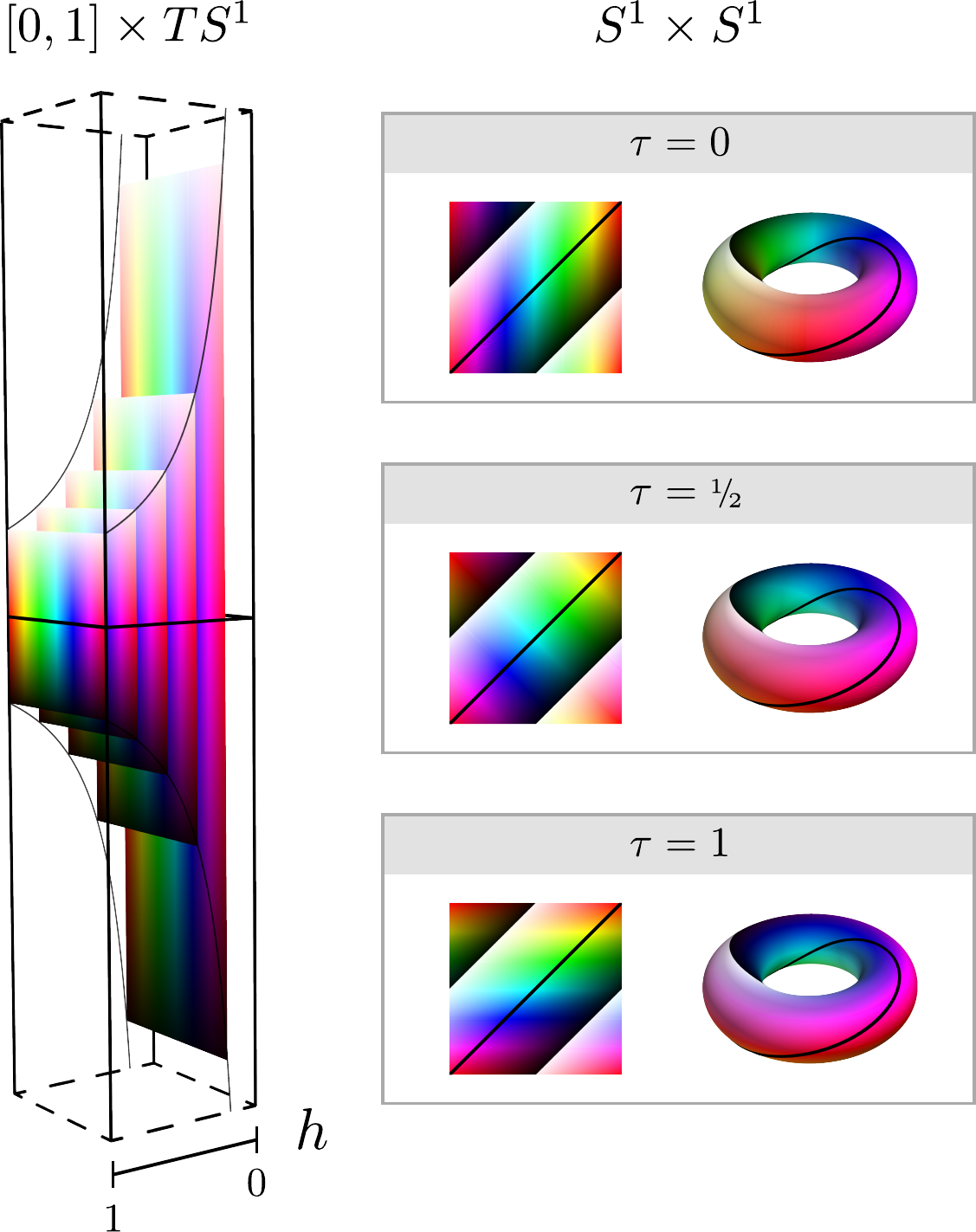}} & \raisebox{-0.5\height}{\includegraphics[width=0.4\linewidth]{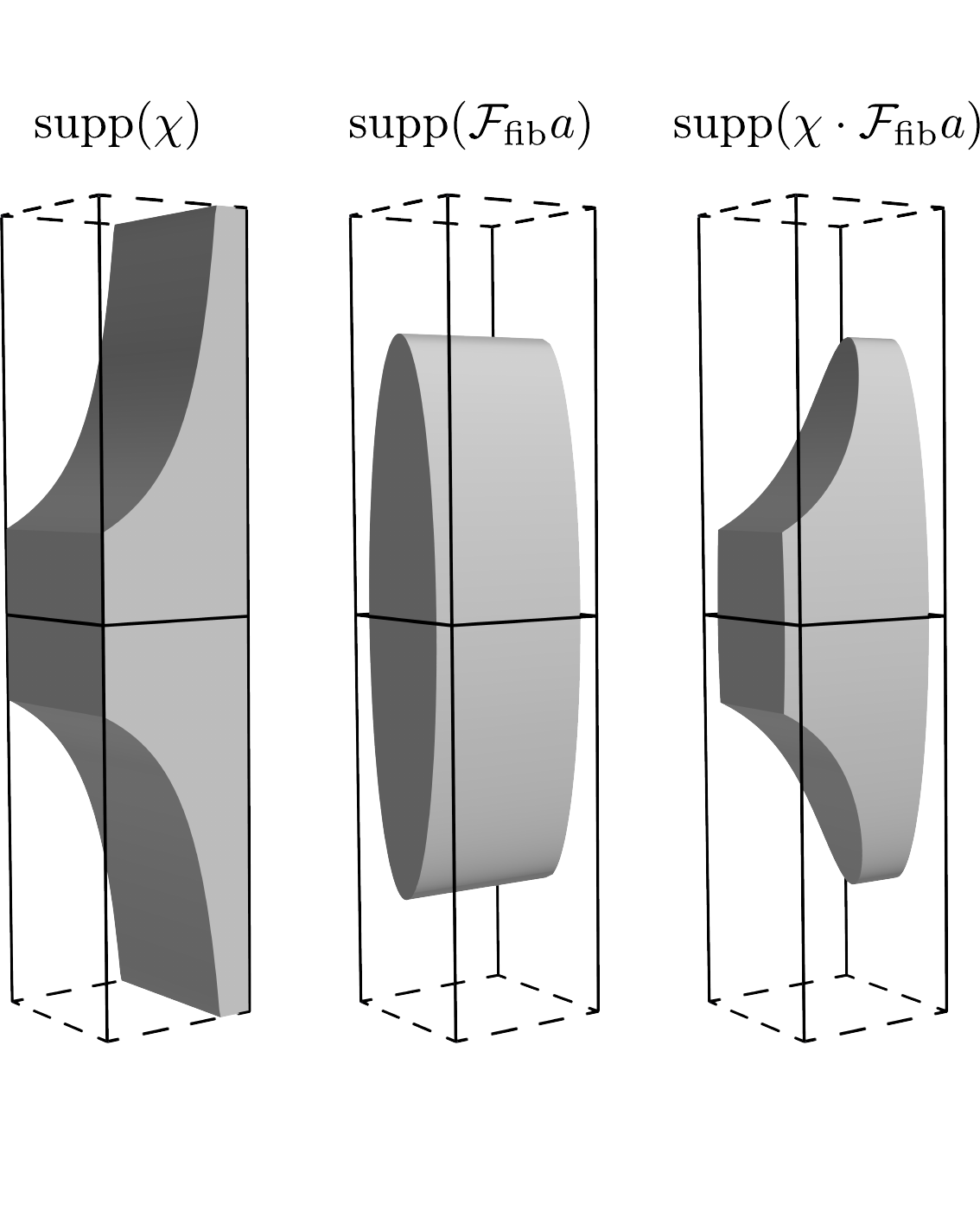}} \\ \small (a) An illustration of $\beta^\tau$ & \small (b) Cut-off procedure
  \end{tabular}
  \caption{Figure (a) is an illustration of $\beta^\tau$ for $M = S^1$ and $\tau \in \{0, \frac12, 1\}$. On the left of figure (a), the cylinder $TS^1$, for each $h \in [0, 1]$, is represented as a vertical strip whose sides are identified. For $h \in \{1, 0.8, 0.6, 0.4, 0.2\}$, the colored rectangle represents $V_\boundary^\tau$ (chosen independently of $\tau$) intersected with the corresponding surface of constant $h$. The zero sections lie in the middle of the figure. On the right of figure (a), we see, for each $\tau$, two representations of $S^1 \times S^1$—as a square with opposing sides identified and embedded in $\R^3$. The maps $\beta^\tau$ are represented by color-matching:\ Each constant-$h$ surface of $V_\boundary^\tau$ on the left is mapped onto $S^1 \times S^1$ as the colors indicate. The zero-sections map onto the diagonal in $S^1 \times S^1$, indicated by black lines. Figure (b) illustrates how the cut-off function $\chi$ scales with $h$, and how we use it to ``restrict'' the function $\fiberF a$ to a smooth function supported in $V_\boundary^\tau$, which we then can pull back via $(\beta^\tau)^{-1}$ to obtain a smooth function on $\TGM$. It also illustrates how the dependence on $\chi$ disappears in the limit $h \to 0$.}\label{fig:tau_quant}
\end{figure}

Geometrically, ignoring the issue of the cut-off function, we can think of $\tau$-quantization as follows:\ Given $a \colon T^*M \to \C$ and $h \in (0, 1]$, the integral kernel of the quantization of $a$ (at parameter value $h$) at $(x, y) \in M \times M$ is obtained by evaluating $\fiberF a \colon TM \to \C$ at the tangent vector to the geodesic going from $y$ to $x$ in time $h$ and scaling by $h^{-d}$. The cut-off function remedies the fact that this geodesic need not be unique. Happily, the operators $\mathcal Q_h^\tau$ become independent of the cut-off function $\chi$ for sufficiently small $h \in (0, 1]$. Indeed, for sufficiently small $h > 0$, we will have $\supp(\fiberF a) \subset \supp(\chi|_h)$, because $\fiberF a$ is compactly supported and $\chi|_h(v_x) = \chi_1(h v_x)$. \Cref{fig:tau_quant} (b) illustrates this phenomenon.

Every $\beta^\tau$ is a diffeomorphism onto its image with respect to the manifold structure on $\TGM$ (which we defined using $\beta = \beta^0$). This implies that $(a, \mathcal Q^\tau_ha) \in C^\infty_c(\TGM) \subset C^*(\TGM)$ for every $a \in \PW(T^*M)$. Landsman also shows that an extension of Weyl quantization (i.e.\ $\tau$-quantization with $\tau = \frac12$) obtained by allowing distributions gives rise to geometric quantization on $T^*M$ \cite[Proposition~3.6.2]{Landsman_Topics}.

\section*{Statements}

\subsection*{AI Statement}

\Cref{cor:all_orders} was added during revision, in response to one of the referees inquiring whether our results lead to higher order expansions and error estimates. The (in hindsight obvious) idea to apply a Taylor expansion to \eqref{eq:proof_of_trace_formula} to obtain \Cref{cor:all_orders} and the application to $\tau$-quantization was suggested by Claude Fable 5.1. The statement and proof was then written from scratch by the authors. All mathematical content of this article besides \Cref{cor:all_orders} was obtained by the authors without any use of AI. No text in this article is AI-generated.
 
\subsection*{Data Availability Statement}

No new data were created or analyzed in this study.

\subsection*{Conflict of Interest Statement}

All authors declare that there is no conflict of interest or competing interests associated to this work.

\subsection*{Funding Statement}

This research received no external funding.

\bibliographystyle{abbrv}
\bibliography{citations} 

\end{document}